\documentclass[reqno]{amsart}
\usepackage{amsmath,amsthm,amsfonts, amssymb,amscd,enumitem}
\usepackage[all]{xy}
\numberwithin{equation}{section}
  \newtheorem{thm}{Theorem}[section]
  \newtheorem{lem}[thm]{Lemma}
  \newtheorem{prop}[thm]{Proposition}
  \newtheorem{cor}[thm]{Corollary}
  \theoremstyle{definition}
  \newtheorem{defn}[thm]{Definition}
  \newtheorem{exm}[thm]{Example}
  \newtheorem{rmk}[thm]{Remark}
  \newtheorem{open}[thm]{Problem}

 \newcommand\ra{\rightarrow}

\newcommand{\lex}{\,\overrightarrow{\times}\,}

 \newcommand\s{\subseteq}

 \newcommand\B{\mathrm{B}}

\newcommand{\id}{\mbox{\rm Id}}
\newcommand\spec{\mathrm{Spec}}
\newcommand\C{\mathrm{C}}

 \numberwithin{equation}{section}

 \def\iff{if and only if }
\def\im{ \mathrm{Im} }

\let\Right\right
\let\Left\left
\makeatletter
\def\right#1{\Right#1\@ifnextchar){\!\right}{}}
\def\left#1{\Left#1\@ifnextchar({\!\left}{}}

\begin{document}
\title[Representation and Embedding of Pseudo MV-algebras with Square Roots I]{Representation and Embedding of Pseudo MV-algebras with Square Roots I. Strict Square Roots}

\author[Anatolij Dvure\v{c}enskij and Omid Zahiri]{Anatolij Dvure\v{c}enskij$^{^{1,2,3}}$, Omid Zahiri$^{^{1,*}}$}

\date{}%
\thanks{The paper acknowledges the support by the grant of
the Slovak Research and Development Agency under contract APVV-20-0069  and the grant VEGA No. 2/0142/20 SAV,  A.D}
\thanks{The project was also funded by the European Union's Horizon 2020 Research and Innovation Programme on the basis of the Grant Agreement under the Marie Sk\l odowska-Curie funding scheme No. 945478 - SASPRO 2, project 1048/01/01,  O.Z}
\thanks{* Corresponding Author: Omid Zahiri}
\address{$^1$Mathematical Institute, Slovak Academy of Sciences, \v{S}tef\'anikova 49, SK-814 73 Bratislava, Slovakia}
\address{$^2$Palack\' y University Olomouc, Faculty of Sciences, t\v r. 17. listopadu 12, CZ-771 46 Olomouc, Czech Republic}
\address{$^3$Depart. Math., Constantine the Philosopher University in Nitra, Tr. A. Hlinku 1, SK-949 01 Nitra, Slovakia}
\email{dvurecen@mat.savba.sk, zahiri@protonmail.com}
\thanks{}



\begin{abstract}
In \cite{DvZa3}, we started the investigation of pseudo MV-algebras with square roots. In the present paper, we continue to study the structure of pseudo MV-algebras with square roots focusing on their new characterizations. The paper is divided into two parts. In the present first part, we investigate the relationship between a pseudo MV-algebra with square root and its corresponding unital $\ell$-group in the scene of two-divisibility.

In the second part, we find some conditions under which a particular class of pseudo MV-algebras can be embedded into pseudo MV-algebras with square roots. We introduce and investigate the concepts of a strict square root of a pseudo MV-algebra and a square root closure, and we compare both notions. We show that each MV-algebra has a square root closure. Finally, using the square root of individual elements of a pseudo MV-algebra, we find the greatest subalgebra of a special pseudo MV-algebra with weak square root.
\end{abstract}

\date{}

\maketitle

\section{Introduction}

Chang \cite{Cha1,Cha2} introduced MV-algebras as the algebraic semantics of the $[0,1]$-valued \L ukasiewicz logic.
In the same sense, Boolean algebras do algebraic semantics of the classical two-valued logic.
Since then, the theory of MV-algebras has been deeply investigated and studied in the different aspects of mathematics and logic.
MV-algebras form a category that is categorically equivalent
to the category of Abelian unital $\ell$-groups. This principle result of the theory was presented by Mundici, \cite{Mun}.

More than twenty years ago, Georgescu and Iorgulescu, \cite{georgescu}, introduced a non-commutative generalization of
MV-algebras called pseudo MV-algebras.  They form an algebraic counterpart of the non-commutative \L ukasiewicz logic; see, for example, \cite{Haj}.
This concept also was independently defined by Rach\r{u}nek \cite{Rac} as generalized MV-algebras. Pseudo MV-algebras have a disjunction that is not necessarily commutative and two negations, the left and right ones, that could coincide even in non-commutative pseudo MV-algebras. We note that other generalizations of pseudo MV-algebras are, for example, pseudo EMV-algebras introduced in \cite{DvZa0,DvZa1} and generalized MV-algebras from \cite{GaTs}.

Dvure\v{c}enskij, \cite{Dvu1}, using Bosbach's notion of a semiclan, presented a categorical equivalence between the category of pseudo MV-algebras and the category of unital $\ell$-groups (not necessarily Abelian ones) that generalizes a similar result by Mundici for MV-algebras. He showed that every pseudo MV-algebra is always an interval in a unital $\ell$-group. Due to Komori \cite{Kom}, it is well-known that the lattice of subvarieties of MV-algebras is countable. In contrast, the lattice of subvarieties of pseudo MV-algebras is uncountable; see \cite{DvHo}. Hence, the structure of pseudo MV-algebras is much richer than that of MV-algebras. Important varieties of pseudo MV-algebras are the variety of symmetric pseudo MV-algebras, where the left and right negations coincide, and the variety of representable pseudo MV-algebras; they will be the playground for us where we will mainly work on the present paper.

Root operators are useful tools in studying algebraic structures with non-idempotent binary operations.
The study of the existence and uniqueness of roots (in a particular case, square root) of an element or all elements of an algebraic structure with binary operations has a long history.
Mal'cev's well-known result, \cite[Thm 7.3.2]{KoMe}, showed that the extraction of roots is unique in a torsion-free locally nilpotent group. So, the equation $x^n=y^n$ has at most one solution. Kontorovi\v{c}, \cite{Kon1,Kon2}, also studied groups with unique extraction of roots called R-groups.
Baumslag in \cite{Bau} developed the theory of groups with unique roots.

In \cite{Hol}, H{\"o}hle studied square roots on integral, commutative, residuated $\ell$-monoids, and especially for MV-algebras. He also introduced a strict square root and proposed a classification of MV-algebras with square roots. In particular, he showed that each MV-algebra with square root is exactly a Boolean algebra (the square root is an identity map), or a strict MV-algebra (the square root is strict), or isomorphic to the direct product of the first two cases. He also investigated $\sigma$-complete, complete,
divisible, locally finite, and injective MV-algebras with square roots and provided their relations.
B\v{e}lohl\'avek, \cite{Bel}, continued in the study of the square root of residuated lattices, and he
proved that the class of all residuated lattices with square roots is a variety. In \cite{Amb}, Ambrosio provided a study about 2-atomless MV-algebras and strict MV-algebras and proved that
on this structure, the concepts of strict and 2-atomless are equivalent. She also proved that each strict MV-algebra contains a copy of the MV-algebra of all rational dyadic numbers. We refer to \cite{Hol,Amb,NPM} for more details about square roots on MV-algebras.

In \cite{DvZa2}, we recently studied square roots on EMV-algebras which generalize MV-algebras. We used square roots to characterize EMV-algebras and provided a representation of EMV-algebras with square roots. Then we investigated divisible and locally complete EMV-algebras with square roots.
It was shown that each strict EMV-algebra has a top element, so it is an MV-algebra.

In the next step, \cite{DvZa3}, we initiated a deed study of square roots on a class of pseudo MV-algebras.
We introduced and investigated the notion of a square root and a weak square root on a pseudo MV-algebra.
The class of pseudo MV-algebras with square roots is equational, so it is a subvariety of
pseudo MV-algebras. We found that for each square root $r$ on a pseudo MV-algebra $M$, the element $r(0)$
plays a significant role. It helped us classify the pseudo MV-algebras class with square roots and proposed several examples. We found a relationship between
two-divisible $\ell$-groups and representable symmetric pseudo MV-algebras with strict square roots.

The present work focuses on investigating square roots, and we extend our research on pseudo MV-algebras, which was initiated in \cite{DvZa3}. The main aims of the present paper, which is divided into two parts, are:

Part I.
\begin{itemize}
\item Present new characterizations and presentations of pseudo MV-algebras with strict and non-strict square roots.
\item Show how two-divisibility is connected with the existence of square roots.
\item Investigate when the two-divisibility of a pseudo MV-algebra $M=\Gamma(G,u)$ entails the two-divisibility of $G$.
\item Study the possibility of embedding a pseudo MV-algebra into a pseudo MV-algebra with square root.
\item Characterize square square roots on strongly $(H,1)$-perfect pseudo MV-algebras.
\end{itemize}

Part II, \cite{DvZa5}.
\begin{itemize}
\item Define and study the strict square root closure of a pseudo MV-algebra.
\item Define and study the square root closure of a pseudo MV-algebra and compare it with the strict square root closure.
\item Investigate the square root of not all elements of a pseudo MV-algebra.
\item Find conditions when a maximal subalgebra exists in a pseudo MV-algebra with weak square root.
\end{itemize}

The paper is organized as follows. Part I. Section 2 gathers basic definitions, properties, and results about pseudo MV-algebras and square roots that will be used in the next sections. Section 3 presents sufficient and necessary conditions under which
a pseudo MV-algebra has a strict or non-strict square root. In Section 4, the relation between a pseudo MV-algebra
$M=\Gamma(G,u)$ with a square root and the two-divisibility of the unital $\ell$-group $G$ is investigated.
If $M$ is linearly ordered or an MV-algebra with square root, then the $\ell$-group $G$ is two-divisible.
We find a sufficient and necessary condition under which a square root on a pseudo MV-algebra $\Gamma(G,u)$ implies the two-divisibility of the $\ell$-group $G$. We also characterize the lexicographic product of MV-algebras with square roots.

Part II.
In Section 5, we find an answer to the question, ``Is it possible to embed a pseudo MV-algebra in a pseudo-MV-algebra with square root?'' To answer the problem, we study a class of pseudo MV-algebras that can be embedded into a pseudo MV-algebra with strict square root. We define the concept of the strict square root closure and prove that each MV-algebra has a strict square root closure. Section 6 introduces a square root closure of an MV-algebra and compares it with the strict square root. Section 7 describes a square root of an individual element of a pseudo MV-algebra and finds the greatest subalgebra of special pseudo MV-algebras with the weak square root property. The paper contains interesting examples illustrating our results and some questions are formulated.

\section{Preliminaries}

In the section, we gather basic elements of $\ell$-groups, pseudo MV-algebras, and square roots.

We will use groups $(G;+,0)$ written additively. A group $G$ is {\it partially ordered} if there is a partial order $\le$ on $G$ such that $f\le g$ implies $h_1+f+h_2\le h_1+g+h_2$ for all $h_1,h_2\in G$. If $\le$ is a lattice order, $G$ is said to be a {\it lattice ordered} or an $\ell$-group. An element $u\ge 0$ is said to be a {\it strong unit} of $G$ if given $g\in G$, there is an integer $n\ge 1$ such that $g\le nu$. A couple $(G,u)$, where $u$ is a fixed strong unit of $G$, is said to be a {\it unital} $\ell$-{\it group}. If $G$ is an $\ell$-group, $\C(G)=\{g\in G\colon g+h=h+g,\ \forall h \in G\}$ is the {\it commutative center} of $G$.

For more information about$\ell$-groups, we recommend to consult with e.g., \cite{Dar,Gla,Anderson}.

\begin{defn} \cite{georgescu}
A {\it pseudo MV-algebra} is an algebra $(M;\oplus,^-,^\sim,0,1)$ of type $(2,1,1,0,0)$ such that the following axioms hold
for all $x,y,z\in M$,

\begin{itemize}
\item[{\rm (A1)}] $x\oplus(y\oplus z)=(x\oplus y)\oplus z$,

\item[{\rm (A2)}] $x\oplus 0=0\oplus x=x$,

\item[{\rm (A3)}] $x\oplus 1=1\oplus x=1$,

\item[{\rm (A4)}] $1^{-}=1^{\sim}=0$,

\item[{\rm (A5)}] $(x^{-}\oplus y^{-})^{\sim}=(x^{\sim}\oplus y^{\sim})^{-}$,

\item[{\rm (A6)}] $x\oplus (x^{\sim}\odot y)=y\oplus (y^{\sim}\odot x)=(x\odot y^{-})\oplus y=(y\odot x^{-})\oplus x$,

\item[{\rm (A7)}] $x\odot (x^{-}\oplus y)=(x\oplus y^{\sim})\odot y$,

\item[{\rm (A8)}] $(x^{-})^{\sim}=x$,
\end{itemize}
where $x\odot y=(x^{-}\oplus y^{-})^{\sim}$. If the operation $\oplus$ is commutative, equivalently $\odot$ is commutative, then $M$ is an MV-algebra. (A6) defines $x\vee y$ and (A7) $x\wedge y$.  We note that if $x^\sim=x^-$ for each $x\in M$, then $\oplus$ is not necessarily commutative. If $x^-=x^\sim$ for each $x\in M$, $M$ is said to be {\it symmetric}.
\end{defn}

We note that it can happen that $0=1$; in this case, $M$ is said to be {\it degenerate}.

For example, if $(G,u)$ is a unital $\ell$-group, then $\Gamma(G,u)=([0,u];\oplus, ^-,^\sim,0,u)$ is a pseudo MV-algebra, where $x\oplus y:=(x+y)\wedge u$, $x^-:=u-x$, and $x^\sim:= -x+u$. Moreover, due to a basic representation of pseudo MV-algebras, see \cite{Dvu1}, every pseudo MV-algebra is isomorphic to some $\Gamma(G,u)$ for a unique (up to isomorphism) unital $\ell$-group $(G,u)$. In addition, the functor $\Gamma: (G,u)\mapsto \Gamma(G,u)$ defines a categorical equivalence between the category of unital $\ell$-groups and the category of pseudo MV-algebras. For more information about the functor $\Gamma$ and its inverse $\Psi$, see \cite{Dvu1}.

According to \cite{Dvu1}, we introduce a partial operation $+$ on any pseudo MV-algebra $M$: Given $x,y\in M$, $x+y$ is defined if and only if $y\odot x=0$, and in such a case, we set $x+y:=x\oplus y$. The operation $+$ is associative, and using the $\ell$-group representation, it corresponds to the group addition in the representing unital $\ell$-group.

For any integer $n\ge 0$ and any $x\in M$, we can define
\begin{align*}
0.x &=0, \quad  \text{and}\quad n.x=(n-1).x\oplus x,\quad n\ge 1,\\
0x&=0, \quad \text{and}\quad nx= (n-1)x+x,\quad n\ge 1,
\end{align*}
assuming $(n-1)x$ and $(n-1)x+x$ are defined in $M$.
An element $x\in M$ is a {\em Boolean element} if $x\oplus x=x$. The set $\B(M)$ denotes the set of Boolean elements of $M$, which is a  subalgebra of $M$ and a Boolean algebra; it is a so-called Boolean skeleton of $M$.

In each pseudo MV-algebra $(M;\oplus,^-,^\sim,0,1)$, we can define two additional binary operations $\to$ and $\rightsquigarrow$ by
$$
x\to y:=x^-\oplus y,\quad x\rightsquigarrow y:=y\oplus x^\sim.
$$
The properties of the operations $\to$ and $\rightsquigarrow$ can be found e.g. in \cite[Prop 2.3]{DvZa3}.

Let $n\ge 2$ be an integer. A pseudo MV-algebra $M$ is said to be $n$-{\it divisible} if, given an element $y\in M$, there is an element $x\in M$ such that $nx$ exists in $M$ and $nx=y$. If $M$ is $n$-divisible for each $n\ge 1$, we say that $M$ is {\it divisible}. We note that $M$ is $n$-divisible iff given $x\in M$, there is $y\in M$ such that $n.y=x$ and $(n-1).y\odot y^-=0$.
Analogously, an $\ell$-group $G$ is $n$-{\it divisible} if given $g\in G$, there is $h\in G$ such that $nh=g$, and $G$ is {\it divisible} if it is $n$-divisible for each $n\ge 2$.
If $M=\Gamma(G,u)$ is an MV-algebra, then $M$ is $n$-divisible iff $G$ is $n$-divisible. For pseudo MV-algebras $\Gamma(G,u)$, $n$-divisibility of $G$ trivially implies $n$-divisibility of $M=\Gamma(G,u)$. The converse also holds if, e.g., $G$ is linearly ordered and $u/2\in \C(G)$; see Corollary \ref{co:strict}.

An $\ell$-group $G$ enjoys {\it unique extraction of roots} if for every integer $n\ge 1$ and  $g,h\in G$, $ng=nh$ implies $g=h$. In the same way, we say that a pseudo MV-algebra enjoys unique extraction of roots. We note that every linearly ordered group, \cite[Lem 2.1.4]{Gla}, (linearly ordered pseudo MV-algebra) enjoys unique extraction of roots. The same applies to  each representable $\ell$-group, see \cite[Page 26]{Anderson}.

\begin{rmk}\label{interval-rmk}
If $(M;\oplus,^-,^\sim,0,1)$ is a pseudo MV-algebra and $a\in M$,
then $([0,a];\oplus_a,^{-_a},^{\sim_a},0,a)$  is a pseudo MV-algebra,
where $x\oplus_a y=(x\oplus y)\wedge a$, $x^{-a}=a\odot x^-$ and $x^{\sim a}=x^\sim\odot a$.
Indeed, by \cite{Dvu1}, we can assume that $M=\Gamma(G,u)$, where $(G,u)$ is a unital $\ell$-group.
For each $a\in M$ and each $x,y\in [0,a]$, $a\odot x^-=a-u+(u-x)\odot a=a-x=x^{-a}$,
$x^\sim\odot a=-x+u-u+a=-x+a=x^{\sim a}$ and
$(x\oplus y)\wedge a=(x+y)\wedge u\wedge a=(x\oplus_a y)$.
Therefore, by \cite[Exm 1.3]{georgescu}, $([0,a];\oplus_a,^{-a},^{\sim a},0,a)$ is a pseudo MV-algebra.
\end{rmk}

A non-empty subset $I$ of a pseudo MV-algebra $M$ is called an {\it ideal} of $M$ if (1) for each $y\in M$, $y\leq x\in I$ implies that $y\in I$; (2) $I$ is closed under $\oplus$. An ideal $I$ of $M$ is said to be (i) {\it prime} if $x\wedge y \in I$ implies $x \in I$ or $y \in I$; (ii) {\it normal} if $x\oplus I=I\oplus x$ for any $x \in M$, where $x\oplus I:=\{x\oplus i \mid i \in I\}$ and $I\oplus x =\{i\oplus x \mid i \in I\}$, (iii) {\it proper} if $I\ne M$,  (iv) {\it maximal} if $I$ is a proper ideal of $M$ and it is not properly contained in any proper ideal of $M$. We denote by $\text{MaxI}(M)$ and $\text{NormI}(M)$ the set of maximal ideals and normal ideals, respectively, of $M$. Of course, $\text{MaxI}(M)\ne \emptyset$ and $\text{NormI}(M)\ne\emptyset$, but their intersection may be empty, see, e.g., \cite{DvHo}, what for MV-algebras is impossible.

We recall that an ideal $I$ is normal \iff given $x,y \in M$,  $x\odot y^-\in I$ \iff $y^\sim \odot x \in I$, \cite[Lem 3.2]{georgescu}. For each subset $I$ of $M$, we denote $I^- =\{x^- \mid  x\in I\}$ and $I^\sim =\{x^\sim \mid  x\in I\}$.

The set of the prime ideals of $M$ is denoted by $\spec(M)$. Equivalent conditions, \cite[Thm 2.17]{georgescu}, for an ideal $I$ to be prime are as follows:
\begin{itemize}[nolistsep]
\item[{\rm (P1)}]  $x\odot y^- \in I$ or $y\odot x^-$ for all $x,y \in M$.
\item[{\rm (P2)}] $x\odot y^\sim \in I$ or $y\odot x^\sim$ for all $x,y \in M$.
\end{itemize}
A one-to-one relationship exists between congruences and normal ideals of a pseudo MV-algebra, \cite[Cor. 3.10]{georgescu}: If $I$ is a normal ideal of a pseudo MV-algebra, then the relation $\sim_I$, defined by $x\sim_I y$ \iff $x\odot y^-, y\odot x^-\in I$, is a congruence,
and $M/I$ with the following operations induced from $M$ is a pseudo MV-algebra, where $M/I=\{x/I\mid  x\in M\}$ and $x/I$ is the equivalence class containing $x$.
\begin{eqnarray*}
x/I\oplus y/I=(x\oplus y)/I,\quad (x/I)^-=x^-/I,\quad (x/I)^\sim=x^\sim/I,\quad 0/I,\quad 1/I.
\end{eqnarray*}
Conversely, if $\theta$ is a congruence on $M$, then $I_\theta =\{x \in M \mid x\sim 0\}$ is a normal ideal such that $\sim_{I_\theta}=\sim$.


A pseudo MV-algebra $M$ is {\it representable} if $M$ is a subdirect product of a linearly ordered pseudo MV-algebras system.
By \cite[Prop. 6.9]{Dvu2}, $M$ is representable \iff $a^\bot=\{x \in M\mid x\wedge a =0\}$ is a normal ideal of $M$ for each $a \in M$. Moreover, the class of representable pseudo MV-algebras forms a variety \cite[Thm 6.11]{Dvu2}.

We will also use pseudo MV-algebras of the form $\Gamma(H\lex G,(u,0))$, where $(H,u)$ is a linearly ordered unital $\ell$-group, $G$ is an $\ell$-group, and $\lex$ denotes the lexicographic product of $H$ and $G$. We note that $\Gamma(H,u)\cong \Gamma(H\lex O,(u,0))$, where $O$ is the one-element zero group, i.e. $O=\{0\}$.

An important class of pseudo MV-algebras is the class of $(H,1)$-perfect pseudo MV-algebras, where $(H,1)$ is a unital $\ell$-subgroup of the unital group of real numbers $(\mathbb R,1)$. They are roughly speaking isomorphic to $\Gamma(H\lex G,(1,0))$, where $G$ is an $\ell$-group, see \cite[Thm 4.3]{Dvu4}, see also \cite{Dvu3}. We recall that if $A,B$ are two subsets of $M$, then $A\le B$ means $a\le b$ for each $a\in A$ and each $b \in B$.
We say a pseudo MV-algebra $M$ is $H$-{\it perfect}, if there is a decomposition $(M_h\mid  h\in H)$ of $M$ such that
\begin{itemize}
\item[{\rm (i)}] $M_s\ne \emptyset$ for each $s \in H$,
\item[{\rm (ii)}] $M_s\le M_t$ if $s\le t$, $s,t \in H$,
\item[{\rm (iii)}] $M_s^-= M_{1-s}= M^\sim_s$ for each $s\in H$,
\item[{\rm (iv)}] if $x\in M_s$ and $y\in M_t$, then $x\oplus y\in M_{s\oplus t}$, where $s\oplus t= \min\{s+y,1\}$, $s,t \in H$.
\end{itemize}

An $(H,1)$-perfect pseudo MV-algebra $M=\Gamma(K,v)$ is {\it strongly} $(H,1)$-{\it perfect}
if there is a system of elements $(c_t\mid  t \in H)$ such that
\begin{itemize}
\item[{\rm (i)}] $c_t \in H\cap \C(K)$, $t \in H$,
\item[{\rm (ii)}] if $s,t \in H$ and $s+t\in H$, then $c_s+c_t=c_{s+t}$,
\item[{\rm (iii)}] $c_1=1$.
\end{itemize}
According to \cite[Thm 4.3]{Dvu4}, a pseudo MV-algebra $M$ is strongly $(H,1)$-perfect iff there is an $\ell$-group $G$ such that $M\cong \Gamma(H\lex G,(1,0))$.

A notion of a square root on MV-algebras was introduced in \cite{Hol}. For pseudo MV-algebras, it was introduced and studied in \cite{DvZa3}.

\begin{defn}
A mapping $r:M\to M$ is said to be (i) a {\em square root} if it satisfies the following conditions:
\begin{itemize}
\item[{\rm (Sq1)}] for all $x\in M$, $r(x)\odot r(x)=x$,

\item[{\rm (Sq2)}] for each $x,y\in M$, $y\odot y\leq x$ implies $y\leq r(x)$,

\item[{\rm (Sq3)}] for each $x\in M$, $r(x^-)= r(x)\to r(0)$ and $r(x^\sim)=r(x)\rightsquigarrow r(0)$,
\end{itemize}
and (ii) a {\it weak square root} if it satisfies only (Sq1) and (Sq2).
A pseudo MV-algebra $(M;\oplus,^-,^\sim,0,1)$ has {\em square roots} ({\it weak square roots}) if there exists a square root (weak square root) $r$ on $M$. If $M$ is an MV-algebras, then both notions of a square root coincide. A square root $r$ is {\it strict} if $r(0)=r(0)^-$, equivalently, $r(0)=r(0)^\sim$. If $M$ has a square root (weak square root), it is unique.
\end{defn}

We note that it can happen that a pseudo MV-algebra has a weak square root but no square root; see \cite[Sec 6]{DvZa3} for such examples.

The basic properties of square roots on pseudo MV-algebras can be found in the following result:

\begin{prop}\label{3.2.0}\cite{DvZa3}
Let $r$ be a square root on a pseudo MV-algebra $(M;\oplus,^-,^\sim,0,1)$. For each $x,y\in M$, we have:
\begin{itemize}[nolistsep]
\item[{\rm (1)}] $x\leq x\vee r(0)\leq r(x)$, $r(1)=1$, $(r(x)\odot r(0))\vee(r(0)\odot r(x))\leq x$ and $r(x)\odot x=x\odot r(x)$.
\item[{\rm (2)}] $x\leq y$ implies that $r(x)\leq r(y)$.
\item[{\rm (3)}] $x\wedge y\leq r(x)\odot r(y), r(y)\odot r(x)$ and if $a\in \B(M)$ such that $a\leq r(0)$, then $a=0$.
\item[{\rm (4)}] $x\leq r(x\odot x)$ and $r(x\odot x)\odot r(x\odot x)=r(x)\odot r(x)\odot r(x)\odot r(x)=x\odot x$.
\item[{\rm (5)}] $(x\wedge x^-)\vee (x\wedge x^\sim)\leq r(0)$.
\item[{\rm (6)}] $r(x)\in \B(M)$ \iff $r(x)=x$.
\item[{\rm (7)}] $r(x)\wedge r(y)=r(x\wedge y)$.
\item[{\rm (8)}] $r(x)\ra r(y)\le r(x\ra y)$ and $r(x)\rightsquigarrow r(y)\le r(x\rightsquigarrow y)$. Moreover, $r(x)\odot r(y)\leq r(x\odot y)$ for all $x,y\in M$ if and only if $r(x)\ra r(y)= r(x\ra y)$ and $r(x)\rightsquigarrow r(y)= r(x\rightsquigarrow y)$.

\item[{\rm (9)}] $r(x\vee y)=r(x)\vee r(y)$.

\item[{\rm (10)}]
$r(x\odot y)\leq (r(x)\odot r(y))\vee r(0)$ and $r(x\odot x)=(r(x)\odot r(x))\vee r(0)$. Consequently, if $r(0)\leq x$, then $r(x\odot x)=x$.

\item[{\rm (11)}] {\rm (a)} $x\in \B(M)$ \iff $r(x)=x\oplus r(0)$ \iff $r(x)=r(0)\oplus x$.

\noindent
{\rm (b)} $(r(0)\ra 0)\odot (r(0)\ra 0)=(r(0)\rightsquigarrow 0)\odot (r(0)\rightsquigarrow 0)\in \B(M)$.

\noindent
{\rm(c)} If $a,b \in M$, $a\le b$, then $r([a,b])=[r(a),r(b)]$. In particular, $r(M)=[r(0),1]$.

\end{itemize}
Properties {\rm (1)--(8)} hold also for each weak square root on $M$.
\end{prop}

\section{Characterizations of pseudo MV-algebras with square roots}

Recently, in \cite{DvZa3}, we presented some characterizations of pseudo MV-algebras with square roots using unital $\ell$-groups.
In this section, we study pseudo MV-algebras and find new conditions under which they have square roots (strict or non-strict).

First, we gather some useful properties that will be used in the sequel.

\begin{defn}\label{7.0}
Let $(M;\oplus,^-,^\sim,0,1)$ be a pseudo MV-algebra and $a$ be an element of $M$. We say that $M$ is

(i) {\it $f$-isomorphic to $[0,a]$}, if there is a bijective and order-preserving map $f:M\to [0,a]$ such that $[0,a]$
with, $\oplus_f$ (the binary operation is inherited from $M$ by $f$, that is, $f(x)\oplus_f f(y):=f(x\oplus y)$), $^{-a}$ and $^{\sim a}$ form a pseudo MV-algebra such that $M\cong [0,a]$ and for all $x\in M$, $f(x)\oplus_f f(x)=(f(x)\oplus f(x))\wedge a$.

(ii) isomorphic to $[0,a]$ if $(M;\oplus,^-,^\sim,0,1)$ is isomorphic to $([0,a];\oplus_a,^{-a},^{\sim a},0,a)$ (see Remark \ref{interval-rmk}).

Note that in (ii), $\oplus_a$ is induced from $\oplus$, i.e. $x\oplus_a y= (x\oplus y)\wedge a$.
Clearly, if $M$ is isomorphic to $[0,b]$ for some $b\in M$, then there is an isomorphism $g:M\to [0,b]$ of pseudo MV-algebras and so $M$ is $g$-isomorphic to $[0,b]$.
\end{defn}

\begin{prop}\label{7.1}
Let $(M;\oplus,^-,^\sim,0,1)$ be a pseudo MV-algebra with a square root $r:M\to M$.
\begin{itemize}[nolistsep]
\item[{\rm (i)}] $[r(0),r(0)^-]\s r(\B(M))$.
\item[{\rm (ii)}] $M$ is a Boolean algebra if $[r(0),r(0)^-]=M$, and $r$ is strict if $[r(0),r(0)^-]=\{r(0)\}$.
\item[{\rm (iii)}] For each $x\in M$, there exists a unique element $R(x)\leq r(0)^-$ such that $R(x)\oplus r(0)=r(x)$.
\item[{\rm (iv)}] $M$ is $f$-isomorphic to $[0,r(0)^-]$ for some bijection $f:M\to [0,r(0)^-]$.
\item[{\rm (v)}] For each $x\in M$, there exists a unique element $R(x)\leq r(0)^-$ such that $R(x)\oplus R(x)=x$.
\item[{\rm (vi)}] If $r(x)\odot r(y)\leq r(x\odot y)$ for all $x,y\in M$, then $M$ is isomorphic to $[0,r(0)^-]$.
\item[{\rm (vii)}] For each $x\in M$, $r(x)=\max\{y\wedge (y\ra x)\mid y\in M\}=\max\{y\wedge (y\rightsquigarrow x)\mid y\in M\}$.
\item[{\rm (viii)}] $x\odot r(0)\leq x\odot x$ for all $x\in M$.
\end{itemize}
\end{prop}

\begin{proof}
Set $v:=r(0)^-\odot r(0)^-$. By \cite[Prop 3.3]{DvZa3} or Proposition \ref{3.2.0}(11), $v\in\B(M)$. Then $r(0)^\sim=r(0)^-$, so that $v=r(0)^\sim\odot r(0)^\sim$.

(i) Let $x\in [r(0),r(0)^-]$. Then $r(0)\leq x$ and \cite[Prop 3.5(i)]{DvZa3} implies
$x=x\vee r(0)=r(x\odot x)$, and so $r(0)^-=r(r(0)^-\odot r(0)^-)=r(v)$.
Also, $x\leq r(0)^-$ implies $x\odot x\leq r(0)^-\odot r(0)^-=v$.
Consider the subset $[0,v]\s M$ which is a subset of $\B(M)$ (\cite[Thm 4.3]{DvZa3}).
By Proposition \ref{3.2.0}(11)(c), $r([0,v])=[r(0),r(v)]=[r(0),r(0)^-]$.
In addition, $[0,v]=\{x\odot x\mid x\in [r(0),r(0)^-]\}$.

(ii) It follows from (i) and \cite[Thm 4.3]{DvZa3}.

(iii) Without loss of generality, we can assume that $M=\Gamma(G,u)$, where $(G,u)$ is a unital $\ell$-group. For each $x\in M$ set $R(x):=r(x^\sim)^-$. By Proposition \ref{3.2.0}(2),  $r(0)\leq r(x^\sim)$ and so $R(x)\leq r(0)^-$. By \cite[Prop 3.5(v)]{DvZa3}, $R(x)\oplus r(0)=r(x)$. Now, let $r(x)=y\oplus r(0)$ for some $y\leq r(0)^-$.
Then $y+r(0)\leq r(0)^-+r(0)=u-r(0)+r(0)=u$ and so $y+r(0)=(y+r(0))\wedge u=y\oplus r(0)$. Similarly,
$R(x)+r(0)=R(x)\oplus r(0)$.
It follows that $R(x)+r(0)=R(x)\oplus r(0)=y\oplus r(0)=y+r(0)$ entails that $R(x)=y$. In addition, from \cite[Prop 3.5(v)]{DvZa3}, we get that $r(x^-)^\sim=R(x)=r(x^\sim)^-$ since $r(x^-)^\sim\leq r(0)^\sim=r(0)^-$ and $r(x^-)^\sim\oplus r(0)=r(0)\oplus r(x^-)^\sim=r(x)$.

(iv) Consider the mapping $f:M\to [0,r(0)^-]$ defined by $f(x)=r(x^\sim)^-$.
From part (iii), we conclude that $f$ is well-defined, one-to-one, and $f(x)=r(x^-)^\sim$ for all $x \in M$.
Also, if $y\leq r(0)^-$, then $r(0)\leq y^\sim$ whence by
\cite[Prop 3.5(i)]{DvZa3}, $y^\sim=y^\sim\vee r(0)=r(y^\sim\odot y^\sim)$. It follows that
$y=r(y^\sim\odot y^\sim)^-=r(z^\sim)^-\in\im(f)$, where $z=y\oplus y$. So, $f$ is a bijection map.
Note that, for each $x\in M$, by \cite[Prop 3.5(v)]{DvZa3}, $r(0)\oplus x=x\oplus r(0)$, consequently
\begin{eqnarray}
\label{E7.1.1} r(0)^-\odot x=x\odot r(0)^-,\quad \forall x\in M.
\end{eqnarray}

Consider the following operations of $[0,r(0)^-]$. If $x,y\in [0,r(0)^-]$, there exist $a,b \in M$ such that $x=f(a)$ and
$y=f(b)$. We set
\begin{eqnarray*}
x\oplus_r y:=f(a\oplus b)=r((a\oplus b)^\sim)^-,\quad \  x^{-r}=x^-\odot r(0)^-,\quad \  x^{\sim r}=x^\sim\odot r(0)^-.
\end{eqnarray*}
Let $x,y,z\in M$.

(1) $f(0)=r(0^\sim)^-=r(1)^-=1^-=0$ and $f(1)=r(1^\sim)^-=r(0)^-$. For each $x,y\in M$,
$f(x)\oplus_r f(y)=r((x\oplus y)^\sim)^-=f(x\oplus y)$.

(2) By (Sq3), \eqref{E7.1.1}, and (iii), we have (recall that in (iii), we have proved that $r(x^-)^\sim=r(x^\sim)^-$ for all $x\in M$)
\begin{align*}
f(x^-)&= r(x^{-\sim})^-=r(x)^-=r(x)^-\wedge r(0)^-=(r(x)^-\oplus r(0))\odot r(0)^-= (r(x)\to r(0))\odot r(0)^-\\
&= r(x^-)\odot r(0)^-=r(x^-)^{\sim-}\odot r(0)^-=r(x^\sim)^{--}\odot r(0)^-=f(x)^-\odot r(0)^-=f(x)^{-r},\\
f(x^\sim)&= r(x)^\sim=r(x)^\sim\wedge r(0)^\sim = r(0)^\sim\odot (r(0)\oplus r(x)^\sim)=r(0)^\sim\odot (r(x)\rightsquigarrow r(0))
=r(0)^\sim \odot r(x\rightsquigarrow 0)\\
&= r(0)^\sim\odot r(x^\sim)=r(0)^\sim \odot r(x^\sim)^{-\sim}=r(0)^\sim\odot f(x)^\sim=f(x)^\sim\odot r(0)^-=f(x)^{\sim r}.
\end{align*}

(3) $f(0)\oplus_r f(x)=r((0\oplus x)^\sim)^-=r(x^\sim)^-=f(x)$. Similarly, $f(x)\oplus_r f(0)=f(x)$.
We have $f(x)\oplus_r f(1)=r((1\oplus x)^\sim)^-=r(1^\sim)^-=r(0)^-=r((x\oplus 1)^\sim)^-=f(1)\oplus_r f(x)$ and
$(r(0)^-)^{\sim r}=(r(0)^-)^{\sim}\odot r(0)^-=r(0)\odot r(0)^-=0$ and
$(r(0)^-)^{- r}=(r(0)^\sim)^{-}\odot r(0)^-=r(0)\odot r(0)^-=0$ (by \cite[Lem 3.15]{DvZa3}).

(4) By the definition of $\oplus_r$, we have $f(x)\oplus_r (f(y)\oplus_r f(z))=f(x)\oplus_r (r((y\oplus z)^\sim)^-)
=f(x)\oplus_r f(y\oplus z)=f(x\oplus (y\oplus z))=f((x\oplus y)\oplus z)$. Similarly,
$(f(x)\oplus_r f(y))\oplus_r f(z)=f((x\oplus y)\oplus z)$. Hence, $\oplus_r$ is associative.

(5) By \cite[Prop 3.5(v)]{DvZa3}, $(f(x)^{-r})^{\sim r}=(f(x)^-\odot r(0)^-)^\sim \odot r(0)^-=(r(0)\oplus f(x))\odot r(0)^-=
(f(x)\oplus r(0))\odot r(0)^-=f(x)\wedge r(0)^-=f(x)$.

(6) By (1) and (2), $(f(x)^{-r}\oplus_r f(y)^{-r})^{\sim r}=(f(x^-)\oplus_r f(y^-))^{\sim r}=f((x^-\oplus y^-)^\sim)$ and
$(f(x)^{\sim r}\oplus_r f(y)^{\sim r})^{-r}=(f(x^\sim)\oplus_r f(y^\sim))^{-r}=f((x^\sim\oplus y^\sim)^-)$.

(7) In a similar way, using (1) and (2), we can show that identities (A7) and (A8) hold.

(8) By Proposition \ref{3.2.0}(10),
$f(x\oplus x)=r((x\oplus x)^\sim)^-=r(x^\sim\odot x^\sim)^-=((r(x^\sim)\odot r(x^\sim))\vee r(0))^-=
(r(x^\sim)^-\oplus r(x^\sim)^-)\wedge r(0)^-=(f(x)\oplus f(x))\wedge r(0)^-$.

Therefore, $([0,r(0)^-];\oplus_r,^{-r},^{\sim r},0,r(0)^-)$ is a pseudo MV-algebras, $f$ is an isomorphism, and $M$ is isomorphic to $[0,r(0)^-]$.

(v) By \cite[Prop 3.5(i)]{DvZa3}, $r(x^\sim)^-\oplus r(x^\sim)^-=x$ and by (iii),
$r(x^\sim)^-\leq r(0)^-$. If $y\leq r(0)^-$ is an element of $M$ such that $y\oplus y=x$, then due to (iv), there exists $z\in M$ such that $r(z^\sim)^-=f(z)=y$.
We have $x=y\oplus y=r(z^\sim)^-\oplus r(z^\sim)^-=z$, and so $y=r(x^\sim)^-$.

(vi) Set $b:=r(0)^-$. Consider the bijection map $f$ defined in the proof of (iv).
By part (iv) and the definition, it suffices to show that $f(x)\oplus_b f(y)=f(x\oplus y)$ for all $x,y\in M$.
Since $r(x)\odot r(y)\leq r(x\odot y)$ and $r(0)\leq r(x\odot y)$, by Proposition \ref{3.2.0}(10), we get
$(r(x)\odot r(y))\vee r(0)= r(x\odot y)$ for all $x,y\in M$. It follows that
\begin{eqnarray*}
f(x\oplus y)&=& r((x\oplus y)^\sim)^-=r((y^\sim\odot x^\sim))^-=\left(\left(r(y^\sim)\odot r(x^\sim)\right)\vee r(0) \right)^-\\
&=& \left(r(x^\sim)^-\oplus r(y^\sim)^-\right)\wedge r(0)^-=(f(x)\oplus f(y))\wedge r(0)^-.
\end{eqnarray*}
Similarly, we can show that if the mapping $f$ is a homomorphism, then $r(x)\odot r(y)\leq r(x\odot y)$ for all $x,y\in M$.

(vii) Let $x\in M$ and $\Omega_x:=\{y\wedge (y\ra x)\mid  y\in M\}$. Then $r(x)\odot r(x)=x$ implies that $r(x)\leq r(x)\ra x$, so $r(x)\wedge (r(x)\ra x)=r(x)$ and $r(x)\in\Omega_x$.

Now, for each $y\in M$, by \cite[Prop 2.3(i)]{DvZa3}, we have
\begin{eqnarray*}
\big(y\wedge (y\ra x) \big)\odot \big(y\wedge (y\ra x) \big)\leq y\odot (y\ra x)=y\wedge x\leq x,
\end{eqnarray*}
which means $y\wedge (y\ra x)\leq r(x)$. Therefore, $r(x)=\max\Omega_x=\max\{y\wedge (y\ra x)\mid  y\in M\}$.

In a similar way, we can show that $r(x)=\max\{y\wedge (y\rightsquigarrow x)\mid  y\in M\}$.

(viii) For each $x\in M$, we have $r(x\odot x)=x\vee r(0)$, by \cite[Prop 3.3(10)]{DvZa3}.
It follows that $(x\odot x)\vee (x\odot r(0))\vee (r(0)\odot x)=(x\vee r(0))\odot (x\vee r(0))=x\odot x$, so
$x\odot r(0)\leq x\odot x$.
\end{proof}

According to Proposition \ref{7.1}(iv), if $M$ is a pseudo MV-algebra with a square root $r$, then $M$ is $f$-isomorphic to $[0,r(0)^-]$,
where $f:M\to [0,r(0)^-]$ is defined by $f(x)=r(x^\sim)^-$ for all $x\in M$.

\begin{thm}\label{7.2}
Let $(M;\oplus,^-,^\sim,0,1)$ be a pseudo MV-algebra. Then $M$ has a square root \iff there exists $b\in M$ satisfying the following conditions:
\begin{itemize}[nolistsep]
\item[{\rm (i)}] $b^-\leq b$ and $b\odot x=x\odot b$ for all $x\in M$,
\item[{\rm (ii)}] $f(x)\oplus f(x)=x$,
\item[{\rm (iii)}] $M$ is $f$-isomorphic to $[0,b]$ for some bijection $f:M\to [0,b]$.
\end{itemize}
In addition, if in condition {\rm (iii)}, $M$ is isomorphic to the $[0,b]$, then $r(x)\odot r(y)\leq r(x\odot y)$ for all $x,y\in M$.
\end{thm}

\begin{proof}
Let $M=\Gamma(G,u)$, where $(G,u)$ is a unital $\ell$-group.

First, assume that there exists $b\in M$ satisfying the conditions (i)--(iii). Define $r:M\to M$ by
\begin{eqnarray*}
r(x)=b^-+f(x), \quad x\in M.
\end{eqnarray*}
Since for each $x\in M$, $b^-+f(x)\leq b^-+b=u$, we have $b^-+f(x)=b^-\oplus f(x)\in M$. In a similar way, $f(x)+b^-=f(x)\oplus b^-$.  Thus, $r$ is one-to-one.
Also, for all $x\in M$, $b^-+f(x)=b^-\oplus f(x)=(f(x)^\sim\odot b)^-=(b\odot f(x)^\sim)^-=f(x)\oplus b^-=f(x)+b^-$.
We claim that $r$ is a square root.

(1) From (i), it follows that $b^\sim=b^-$. Hence, $r(0)=b^-\oplus f(0)=b^-=b^\sim$.

(2) Since $M$ is $f$-isomorphic to $[0,b]$, $f(x^-)^\sim=(f(x)^-\odot b)^\sim=b^\sim\oplus f(x)=b^-\oplus f(x)=r(x)$ for each $x\in M$.

(3) By (ii) and (2), we get that $r(x)\odot r(x)=f(x^-)^\sim\odot f(x^-)^\sim=(f(x^-)\oplus f(x^-))^\sim=(x^-)^\sim=x$.

(4) Since $M$ is $f$-isomorphic to $[0,b]$, property (2) implies that $r(x\ra 0)=r(x^-)=f(x^{--})^\sim=(f(x^-)^-\odot b)^\sim=
b^\sim\oplus f(x^-)=b^-\oplus f(x^-)^{\sim-}=b^-\oplus (f(x^-)^{\sim})^-=r(x)^-\oplus b^-=r(x)\ra r(0)$.

(5) Let $y,x\in M$ such that $y\odot y\leq x$. Then $f(y\odot y)\leq f(x)$.
\begin{eqnarray*}
f(y\odot y)&=& f((y^-\oplus y^-)^\sim)=f(y^-\oplus y^-)^\sim\odot b=\big((f(y^-)\oplus f(y^-))\wedge b\big)^\sim\odot b \\
&=& \left(\left(f(y)^-\odot b \right)\oplus \left(f(y)^-\odot b \right)\right)^\sim\odot b=
\left(f(y)^-\odot b \right)^\sim\odot \left(f(y)^-\odot b \right)^\sim\odot b\\
&=& (r(y)\odot r(y))\odot b=y\odot b \mbox{ by (1), (2), and (3)}.
\end{eqnarray*}
From (2), we conclude that $y\leq y\vee b^-=(y\odot b)\oplus b^-=f(y\odot y)\oplus b^-\leq f(x)\oplus b^-=r(x)$.

(3)--(5) imply $r$ is a square root on $M$.

Conversely, if $M$ has a square root $r$, then
it suffices to set $b:=r(0)^-$ and $f(x)=r(x^\sim)^-$ for all $x\in M$. By Proposition \ref{7.2}(iv),
conditions (i)--(iii) hold.

Now, let $M$ be isomorphic to $[0,b]$. Consider the isomorphism $g:M\to [0,b]$. By the first part and the note right after Definition \ref{7.0},
$r(x)=g(x^-)^\sim=g(x)\oplus b^-$ is a square root on $M$. Let $x,y\in M$.
\begin{eqnarray*}
r(x\odot y)&=& f(y^-\oplus x^-)^\sim=f(y^-\oplus x^-)^\sim=\big((f(y^-)\oplus f(x^-))\wedge b\big)^\sim \mbox{ by the assumption}\\
&=& \big((f(y)^-\odot b )\oplus (f(x)^-\odot b )\big)^\sim\vee b^\sim=
\left(f(x)^-\odot b \right)^\sim\odot \left(f(y)^-\odot b \right)^\sim\vee b^\sim\\
&=& (r(y)\odot r(y))\vee b \mbox{ by (2) and (3)}.
\end{eqnarray*}
Therefore, $M$ satisfies the condition $r(x)\odot r(y)\leq r(x\odot y)$ for all $x,y\in M$.
\end{proof}

\begin{cor}\label{7.3}
Let $M$ be an MV-algebra. Then $M$ has a square root \iff there exists $b\in M$ such that $b^-\leq b$,
$M$ is $f$-isomorphic to the MV-algebra $[0,b]$ for some isomorphism $f:M\to[0,b]$, and for all $x\in M$, $f(x)\oplus f(x)=x$.
\end{cor}

\begin{proof}
It follows from Theorem \ref{7.2}. Note that if $r$ is a square root on $M$, then for each $x,y\in M$,
$r(x)\odot r(y)\odot r(x)\odot r(y)=r(x)\odot r(x)\odot r(y)\odot r(y)=x\odot y$ and so by (Sq2),
$r(x)\odot r(y)\leq r(x\odot y)$.
\end{proof}

\section{Square roots and two-divisibility}

In \cite{DvZa3}, we studied the existence of a square root on a pseudo MV-algebra $M=\Gamma(G,u)$, where the unital $\ell$-group $(G,u)$ is two-divisible. In this section, we study this question in more detail, in particular, we characterize square roots on strongly $(H,1)$-perfect pseudo MV-algebras.

Note that if $(G;+,0)$ is an $\ell$-group and $x,y\in G$ such that $(x+y)/2=x/2+y/2$, then
$(x/2+y/2)+(x/2+y/2)=x+y=(x/2+x/2)+(y/2+y/2)$ and so $y/2+x/2=x/2+y/2$. It yields that
$x+y=y+x$. Hence, if $G$ is two-divisible that enjoys unique extraction of roots, then $G$ satisfies the identity
$(x+y)/2=x/2+y/2$ \iff $G$ is Abelian.

More generally, if $G$ is a two-divisible $\ell$-group that enjoys unique extraction of roots and satisfies the following inequality
\[
x/2+y/2\leq (x+y)/2,\quad  x,y\in G,
\]
then $x/2=((x-y)+y)/2\geq (x-y)/2+y/2$ and so $x/2-y/2\geq (x-y)/2$ for all $x,y\in G$.
Substituting $y$ by $-y$ in the last inequality, we get $x/2+y/2=x/2-(-y)/2\geq (x+y)/2$ for all $x,y\in M$. That is,
\[
x/2+y/2= (x+y)/2,\quad  x,y\in G.
\]
In a similar way, $x/2+y/2\geq (x+y)/2$ for all $x,y\in G$ implies that
$x/2+y/2= (x+y)/2$ for all $x,y\in G$.

\begin{thm}\label{8.1}
Let a pseudo MV-algebra $M=\Gamma(G,u)$ be the direct product of linearly ordered pseudo MV-algebras and let $M$ be with strict square root. Then $G$ is two-divisible.
\end{thm}

\begin{proof}
Let $r$ be a strict root on $M$. We note that by \cite[Thm 5.6]{DvZa3}, $M$ is symmetric, $u/2$ exists and belongs to the center of $\C(G)$ of $G$, and by \cite[Thm 5.2]{DvZa3}, $r(x)=(x+u)/2$ for each $x\in M$. Since $M$ is representable, so is $G$, and by \cite[Page 26]{Anderson}, $G$ enjoys unique extraction of roots.

For each $x\in G$ if $x/2$ exists, then $(x/2+u/2)+(x/2+u/2)=x/2+x/2+u/2+u/2=x+u$ which implies that $(x+u)/2=x/2+u/2$.
Similarly, $(x+nu)/2=x/2+nu/2$.
Also, $-(x/2)-(x/2)=-x$ implies that $(-x)/2$ exists and is equal to $-x/2$. Consequently,
$(x-u)/2=x/2-u/2$.

(I) Let us first assume $M$ is linearly ordered; then so is $G$.
We show that for every $g\in G$, the element $g/2$ is defined in $G$.

(1) If $g\in [0,u]$, then $g\in M$, $r(g)=(g+u)/2$ is defined in $[0,u]$, and $(g+u)/2-u/2= (g+u-u)/2=g/2\in G$.

(2)  Assume $g\in G^+$ is such that $nu\le g \le (n+1)u$ for some integer $n\ge 0$. From
$0\leq g-nu\leq u$ and (1), it follows that $(g-nu)/2$ exists. Hence, $(g-nu)/2+n(u/2)=(g-nu)/2+nu/2=g/2$.

(3) If $g \in G^-$, then $-g\in G^+$ and $(-g)/2=-(g/2)$ is defined in $G$, consequently $g/2$ exists in $G$ and is unique.

Summarizing (1)--(3), $G$ is two-divisible.

(II) Now, let $M=\prod_{i\in I} M_i$, where each $M_i$ is a linearly ordered pseudo MV-algebra. Then, for each $i\in I$, $M_i\cong \Gamma(G_i,u_i)$ with a linearly ordered unital $\ell$-group $(G_i,u_i)$. Without loss of generality, we can assume $M_i=\Gamma(G_i,u_i)$. Define $r_i:M_i\to M_i$ by $r_i(x_i)=\pi_i(r(x))$, $x_i\in M_i$ and $x=(x_j)_j$, where $\pi_i$ is the $i$-th projection from $M$ onto $M_i$. By \cite[Prop 3.9]{DvZa3}, $r_i$ is a strict square root on $M_i$. Due to \cite[Thm 5.2]{DvZa3}, $u_i/2,u_i\in \C(G_i)$ for each $i\in I$. By part (I), every $G_i$ is two-divisible.

We describe the unital $\ell$-group $(G,u)$: Put
$u=(u_i)_i$, then
$$
G=\{g=(g_i)_i\in \prod_{i\in I}G_i\mid  \exists n\in \mathbb N: -nu\le g \le nu\}.
$$
Let $g\in G^+$, then $g=(g_i)_i$, where each $g_i\ge 0$. Therefore, $g_i/2\ge 0$ exists in $G_i$ for each $i\in I$, and  $g/2=(g_i/2)_i$ exists in $G$ while $0\le g/2\le g\le nu$. Now take an arbitrary $g\in G$. There is an integer $n\in \mathbb N$ such that $-nu\le g\le nu$. Then $g+nu\ge 0$ so that $(g+nu)/2$ is defined. Therefore, $(g+nu)/2-n(u/2)=(g+nu-nu)/2=g/2$ is defined in $G$.
\end{proof}

\begin{cor}\label{18-1}
Let $M=\Gamma(G,u)$ be a representable pseudo MV-algebra with strict square root. Then $(G,u)$ can be embedded into a two-divisible unital $\ell$-group.
\end{cor}

\begin{proof}
Let $r$ be a strict square root on $M$.
Let $M_i=\Gamma(G_i,u_i)$ for each $i\in I$ be a linearly ordered pseudo MV-algebra, and $f:M\to M_0:=\prod_{i\in I}M_i$ be a subdirect embedding. Without loss of generality, we can assume that each $M_i$ is non-degenerate. By \cite[Prop 3.9]{DvZa3}, $r_i:M_i\to M_i$, defined by $r_i(\pi_i\circ f(x))= \pi_i\circ f(r(x))$ for all $x\in M$, is a square root on $M_i$.
Since $M_i$ is a chain,  by \cite[Cor 4.5 and 5.7(3)]{DvZa3}, $r_i$ is strict or $|M_i|=2$.
If $|M_i|=2$, then $M_i=\{0,1\}$, so by \cite[Thm 3.8]{DvZa3}, $r_i(0_i)=0_i$. On the other hand, since $r$ is strict we have $
r_i(0_i)=r_i(\pi_i\circ f(0))=\pi_i\circ f(r(0))=\pi_i\circ f(r(0)^-)=(\pi_i\circ f(r(0)))^-=r_i(0_i)^-=1_i$, which means $M_i$ is degenerate, that is absurd. So, $|M_i|\neq 2$ for all
$i\in I$.
It follows from Theorem \ref{8.1}, that for each $i\in I$, $G_i$ is two-divisible. If we define $G_0$ by
\begin{equation}\label{eq:G}
G_0=\{g=(g_i)_i\in \prod_{i\in I}G_i\mid  \exists n\in \mathbb N: -nu_0\le g \le nu_0\},
\end{equation}
where $u_0=(u_i)_i$, then $(G_0,u_0)$ is a two-divisible unital $\ell$-group in which $(G,u)$ can be embedded.
\end{proof}

Theorem \ref{8.1} entails the following question:

\begin{open}\label{op:1}
Does Theorem {\rm \ref{8.1}} hold if $M=\Gamma(G,u)$ is a subdirect product of linearly ordered pseudo MV-algebras?
\end{open}

Now, we show that Theorem \ref{8.1} holds for every MV-algebra with strict square root.

\begin{thm}\label{8.20}
Let $M=\Gamma(G,u)$ be an MV-algebra with a strict square root $r$, where
$(G,u)$ is a unital $\ell$-group. Then $G$ is two-divisible.
\end{thm}

\begin{proof}
By \cite[Thm 5.2]{DvZa3}, $r(x)=(x+u)/2$ for each $x\in M$.
So, $r(u-x)=(u-x+u)/2=x/2$ which means $x/2$ exists for all $x\in M$.
Now, let $g\in G^+$. Then there exists $n\in\mathbb N$ such that $g\leq nu$. The
Riesz interpolation property (see \cite[Thm 1.3.11]{Dar}) implies that
$g=\sum_{i=1}^n x_i$ where $0\leq x_i\leq u$. By the assumption, $x_i/2\in G$ for all $i\in\{1,2,\ldots,n\}$.
It follows that $g=\sum_{i=1}^n x_i=2(\sum_{i=1}^n (x_i/2))$, that is $g/2=\sum_{i=1}^n (x_i/2)\in G$.
Now, choose an arbitrary element $g\in G$.
Then $g=g^+-g^-$. Since $g^+,g^-\in G^+$, the elements $(g^+)/2,(g^-)/2$ exist.
It follows that $g=g^+-g^-=(g^+)/2-(g^-)/2+(g^+)/2-(g^-)/2$. Hence, $g/2$ exists, and $G$ is two-divisible.
\end{proof}

\begin{cor}\label{8.21}
Let $M=\Gamma(G,u)$ be an MV-algebra with a square root $r$, where
$(G,u)$ is a unital $\ell$-group. Then $G$ is isomorphic to the direct product of unital $\ell$-groups $(G_1,u_1)$ and $(G_2,u_2)$,
where $G_1$ is a subdirect product of copies of $(\mathbb Z,1)$, and $G_2$ is a two-divisible $\ell$-group.
\end{cor}

\begin{proof}
By \cite[Thm 2.21]{Hol},  $M\cong M_1\times M_2$, where $M_1$ is a Boolean algebra and $M_2$ is an MV-algebra with a strict square root $s:M_2\to M_2$. Let $M_i=\Gamma(G_i,u_i)$, for $i=1,2$.
Since $M_2$ is strict, by Theorem \ref{8.20}, $G_2$ is two-divisible.
If $X$ is the set of all prime ideal of $M_1$, then $f:M_1\to \prod_{P\in X}M_1/P$, defined by $f(x)=(x/P)_{P\in X}$, is a subdirect embedding.
Clearly, $\prod_{P\in X}M_1/P\cong \prod_{P\in X}\{0,1\}=\prod_{P\in X}\Gamma(\mathbb Z,1)$.
Therefore, $G_1$ can be embedded in $\prod_{P\in X}\mathbb Z$.
\end{proof}

In the following example, we show that Theorem \ref{8.1} does not hold for MV-algebras with non-strict square root.

\begin{exm}\label{8.2}
Let $B=\{0,1\}$ with $0<1$ be a two-element Boolean algebra. By \cite{Hol,DvZa3}, $r:=\id_B$ is a square root on $B$ that is not strict.
Then $M=\Gamma(\mathbb Z,1)$, where $(\mathbb Z,1)$ is the unital $\ell$-group of integers with the strong unit $1$.
Clearly, $\mathbb Z$ is not two-divisible.
\end{exm}

\begin{prop}\label{8.3}
Let $M=\Gamma(G,u)$ be a representable pseudo MV-algebra with a square root $r$. If $G$ is two-divisible, then $r$ is strict.

Conversely, if $M$ is the direct product of linearly ordered pseudo MV-algebras, then $G$ is two-divisible.
\end{prop}

\begin{proof}
Let $G$ be two-divisible. We claim that $w=r(0)^-\odot r(0)^-=0$. By Proposition \ref{3.2.0}(11), $w\in \B(M)$ and by the proof of Case 3 in \cite[Thm 4.3]{DvZa3},
$([0,w];\oplus,^{-_w},^{\sim_w},0,w)$ is a Boolean algebra, so $[0,w]\s \B(M)$. Choose $b\in [0,w]$. Since $G$ is two-divisible, $b/2\in G$ exists, and $G$ is a representable unital $\ell$-group (since $M$ is representable), so $G$ enjoys unique extraction of roots and $0\leq b/2\leq b\leq u$.
It follows that $b/2\in [0,w]$ and so $b/2=b/2\oplus b/2=(b/2+b/2)\wedge u=b\wedge u=b$. Consequently, $b=b/2=0$.
Therefore, $w=0$. Applying the proof of Case 2 in \cite[Thm 4.3]{DvZa3}, we have $r$ is strict.

The converse follows from Theorem \ref{8.1}.
\end{proof}

\begin{cor}\label{co:strict}
Let a pseudo MV-algebra $M=\Gamma(G,u)$ be a direct product of linearly ordered pseudo MV-algebras. The following statements are equivalent:
\begin{itemize}
\item[{\rm (i)}] The pseudo MV-algebra $M$ has a strict square root.
\item[{\rm (ii)}] The pseudo MV-algebra $M$ is two-divisible and $u/2\in \C(G)$.
\item[{\rm (iii)}] The $\ell$-group $G$ is two-divisible and $u/2\in \C(G)$.
\end{itemize}
In either case, $(x+u)/2$ is defined in $M$ for each $x\in M$, and $r(x)=(x+u)/2$, $x\in M$, is a strict square root on $M$.

\end{cor}

\begin{proof}
(i) $\Rightarrow$ (ii). It was established in \cite[Thm 5.6]{DvZa3}.

(ii) $\Rightarrow$ (i). If $x\in M$, then $x/2$ and $u/2$ are defined in $M$, therefore, $(x+u)/2=(x/2)+(u/2)$ exists in $M$, and the mapping $r(x)=(x+u)/2$, $x \in M$, is, in fact, a strict square root on $M$.

(i) $\Rightarrow$ (iii). By the equivalence (i) and (ii), we have $u/2\in \C(G)$. Theorem \ref{8.1} entails implication (i) $\Rightarrow$ (iii).

(iii) $\Rightarrow$ (i). The mapping $r(x)=(x+u)/2$, $x\in M$, is a strict square root on $M$; see \cite[Ex 3.7(iii)]{DvZa3}.
\end{proof}

The following result is a partial answer to the question posed in Problem \ref{op:1} above.

We note that a pseudo MV-algebra $M=\Gamma(G,u)$ is said to be {\em dense} in $G$, if for each $g\in G^+$, there exists $x\in M$ and $n\in \mathbb N$ such that $g=nx$.

\begin{thm}\label{dense2}
Let $M=\Gamma(G,u)$ be a representable pseudo MV-algebra, where
$(G,u)$ is a unital $\ell$-group. The following statements are equivalent:
\begin{itemize}[nolistsep]
\item[{\rm (i)}] The $\ell$-group $G$ is two-divisible and $u/2\in \C(G)$.
\item[{\rm (ii)}] The pseudo MV-algebra $M$ is dense in $G$ and $M$ has a strict square root.
\end{itemize}
\end{thm}

\begin{proof}
Let $g\in G$ and take the positive and negative parts of $g$: $g^+:=g\vee 0$ and $g^-:=-(g\wedge 0)= -g\vee 0$. Then $g^+\wedge g^-=0$ so that $g^++g^-=g^-+g^+$, and
$g+(-g\vee 0)=0\vee g=(-g\vee 0)+g$. That is, $g=g^+-g^-=-g^-+g^+$.

(i) $\Rightarrow$ (ii). Suppose that $G$ is two-divisible and $u/2\in \C(G)$. For each $g\in G^+$, there exists $n\in\mathbb N$ such that $g\leq nu\leq 2^nu$. It follows that
$0\leq g/2^n\leq u$. Hence $M$ is dense in $G$. Consider the mapping $r:M\to M$ defined by $r(x)=(x+u)/2$. By \cite[Exm 3.7(iii)]{DvZa3},
$r$ is a square root on $M$. In addition, $r(0)=u/2=u-r(0)=r(0)^-$, whence $r$ is strict.

(ii) $\Rightarrow$ (i). Assume that $M$ is dense in $G$ and $r:M\to M$ is a strict square root on $M$.
There exist $m,n\in \mathbb N$ such that $g^+=nx$ and $g^-=my$ for some $x,y\in M$.
Since $r$ is strict, by the first step in the proof of Theorem \ref{8.1},
$x/2,y/2\in M$. We have $g^+=nx=n(x/2+x/2)=n(x/2)+n(x/2)$ and $g^-=my=m(y/2+y/2)=m(y/2)+m(y/2)$.
The elements $c=g^+/2$ and $d=g^-/2$ exist in $G$.

On the other hand, since $M$ is representable, $G$ also is representable. Let $h:M\to \prod_{i\in I} M_i$ be a subdirect embedding, where $\{M_i\mid  i\in I\}$ is a family of linearly ordered pseudo MV-algebras. Suppose that $M_i=\Gamma(G_i,u_i)$ for all $i\in I$, where $G_i$ an $\ell$-group with the strong unit $u_i$. Then $\{(G_i,u_i)\mid  i\in I\}$ is a family of linearly ordered unital $\ell$-groups, and let $f:G\to \prod_{i\in I}G_i$ be a subdirect embedding.

We have $f(g^+)=f(g)\vee f(0)=f(g)\vee (0)_{i\in I}$ and $f(g^-)=-f(g)\vee f(0)=-f(g)\vee (0)_{i\in I}$.
Suppose that $f(g)=(g_i)_{i\in I}$, $f(g^+)=(g^+_i)_{i\in I}$, $f(g^-)=(g^-_i)_{i\in I}$, $f(c)=(c_i)_{i\in I}$ and $f(d)=(d_i)_{i\in I}$.
The linearity of $G_i$ gives $g^+_i\neq 0 \Leftrightarrow g_i>0 \Leftrightarrow g^-_i=0$ and $g^-_i\neq 0 \Leftrightarrow g_i<0 \Leftrightarrow g_i^+=0$.
Since $2c_i=g^+_i$,  $2d_i=g^-_i$ and $G_i$ is a chain, $c_i\neq 0\Leftrightarrow g^+_i\neq 0$ and $d_i\neq 0\Leftrightarrow g^-_i\neq 0$ for all $i\in I$. Hence, $c_i\neq 0 \Leftrightarrow d_i=0$, which yields $c_i+d_i=d_i+c_i$ and therefore, we have $c_i-d_i=-d_i+c_i$.
It follows that $f(c-d)=f(c)-f(d)=(c_i)_{i\in I}-(d_i)_{i\in I}=-(d_i)_{i\in I}+(c_i)_{i\in I}=-f(d)+f(c)=f(-d+c)$
consequently, $c-d=-d+c$. Whence $2(c-d)=2c-2d=g^+-g^-=g$. That is, $g/2$ exists in $G$. Therefore, $G$ is two-divisible.
\end{proof}

In \cite[Thm 5.6]{DvZa3}, we showed that if $M=\Gamma(G,u)$ is a representable pseudo MV-algebra with a strict square root $r$, then
$u/2$ exists, belongs to $\C(G)$, and $r(0)=u/2$. A similar result for the general case is as follows:

\begin{prop}\label{8.4}
Let $r$ be an arbitrary square root on a representable pseudo MV-algebra $M=\Gamma(G,u)$ and $w=r(0)^-\odot r(0)^-$.
Then $(u-w)/2$ exists, is equal to $r(0)$, and for all $x\in M$, $x+(u-w)/2=(u-w)/2+x$. In particular, $(u-w)/2=r(0)=(-w+u)/2\in \C(G)$.
\end{prop}

\begin{proof}
Since $w\in \B(M)$, $w\vee r(0)=w\oplus r(0)=(r(0)^-\odot r(0)^-)\oplus r(0)=r(0)^-\vee r(0)=r(0)^-$. Also, by Proposition \ref{3.2.0}(11), $w\odot r(0)=r(0)\odot w=0$. Thus, $w+r(0)=w\oplus r(0)=r(0)^-$ entails that $r(0)^--r(0)=w$. Similarly, $r(0)+w=r(0)\oplus w=r(0)^-$. From $u=r(0)^-+r(0)=w+r(0)+r(0)$ we get that $-w+u=2r(0)$ and so $r(0)=(-w+u)/2$ exists. On the other hand, by \cite[Lem 3.15]{DvZa3}, $u=r(0)+r(0)^\sim=r(0)+r(0)^-=r(0)+r(0)+w$. Thus, $u-w=2r(0)$ and $r(0)=(u-w)/2$. Consequently, $(u-w)/2=r(0)=(-w+u)/2$.

(I) First, assume that $M$ is a chain. Choose $x\in M$.

(1) If $x+(u-w)/2< u$, by \cite[Prop 3.5(vi)]{DvZa3}, we have $x+(u-w)/2=x\oplus (u-w)/2=x\oplus r(0)=r(0)\oplus x=(r(0)+x)\wedge u$. Since $M$ is a chain, we have $r(0)+x\le u$, and whence $x+(u-w)/2=(r(0)+x)\wedge u=r(0)+x=(u-w)/2+x$.

(2) If $x+(u-w)/2=u$, then $x+r(0)=u$ and so $x=u-r(0)=r(0)^-=r(0)^\sim=-r(0)+u=-((u-w)/2)+u$. It follows that $(u-w)/2+x=u$.

(3) If $u<x+(u-w)/2$, then for $y=x-((u-w)/2)$, we have $y+(u-w)/2=x$, so $0\leq y+(u-w)/2\leq u$. Hence, by the first parts of the proof, $x=x-((u-w)/2)+(u-w)/2=y+(u-w)/2=(u-w)/2+y=(u-w)/2+x-((u-w)/2)$ which implies that $x+(u-w)/2=((u-w)/2)+x$.

(II) Now, let $X:=X(M)$ be the set of all proper normal prime ideals of $M$.
Then $r_P:M/P\to M/P$ defined by $r_P(x/P)=r(x)/P$ ($x\in M$) is a square root on $M/P$; see \cite[Prop 3.9]{DvZa3}. Given $P\in X$, let $\hat P$ be the $\ell$-ideal of $G$ generated by $P$. Then, due to the representation theorem of pseudo MV-algebras by unital $\ell$-groups, \cite{Dvu1}, $M/P=\Gamma(G/\hat P, u/P)$, and since $M/P$ is a chain, then $G/\hat P$ is a linearly ordered group.  The map $f:M\to\prod_{P\in X}M/P$, $f(x)=(x/P)_{P\in X}$, is a subdirect embedding, and $\hat f: (G,u)\to \prod_{P\in X}(G/\hat P,u/P)$ is also a subdirect embedding; it is an extension of $f$.
Then $f(r(x))=(r_P(x/P))_{P\in X}$, so that $f(r(0))=(r(0)/P)_{P\in X}$ and  $f(w)= (w/P)_{P\in X} = (w_P)_{P\in X}$, where $w_P= r_P(0/P)^-\odot r_P(0/P)^-$. Due to part (I), we have $(x/P)+_P (u/P-_Pw_P)/2=(u_P-_P w/P)/2+_P (x/P)$ for each $x\in M$, where $+_P$ and $-_P$ are the addition and subtraction, respectively, in $G/\hat P$. Therefore,
\begin{eqnarray*}
\hat f(x+ ((u-w)/2))&=&\hat f(x)+ \hat f((u-w)/2)\\
&=& \big(x/P\big)_{P\in X}+_P\big((u_P-_Pw/P)/2\big)_{P\in X}\\
&=& \big(x/P+_P(u_P-_Pw/P)/2\big)_{P\in X}\\ &=&\big((u/P-_Pw/P)/2+_Px/P\big)_{P\in X}\\
&=&\hat f((u-w)/2)+ \hat f(x)= \hat f((u-w)/2)+ x),
\end{eqnarray*}
giving $x+(u-w)/2=(u-w)/2+x$ for each $x\in M$ as stated.

Since, every $x\in G^+$ is of the form $x=x_1+\cdots+x_n$ for some $x_1,\ldots,x_n \in M$, we have $x+(u-w)/2=(u-w)/2+x$. The equality holds for each $x\in G^-$. If $x=g\in G$, then the equality holds for both $g^+$ and $g^-$, and finally, for each $g\in G$.
\end{proof}

Consequently, the latter proposition says that if $r$ is a square root on a representable pseudo MV-algebra $M=\Gamma(G,u)$, the element $r(0)=(u-w)/2=(-w+u)/2 \in \C(G)$. If $r$ is strict, equivalently, $w=0$, then $r(0)=u/2\in \C(G)$ as it was established in \cite[Thm 5.6]{DvZa3}, and Proposition \ref{8.4} generalizes the situation known only for strict square roots.

In \cite[Prob 4.7]{DvZa3}, we proposed a question whether the class of pseudo MV-algebras with square roots satisfying
$r(x)\odot r(x)\leq r(x\odot y)$ for all $x,y\in M$ is a proper subvariety of the variety of pseudo MV-algebras with square roots. In the sequel, we give a partial answer to it.

\begin{prop}\label{8.5}
Let $\mathcal V$ be the variety of pseudo MV-algebras with square roots satisfying the inequality
\begin{eqnarray}
\label{Eq8.5} r(x)\odot r(y)\leq r(x\odot y).
\end{eqnarray}
Then $\mathcal V$ properly contains the variety $\mathcal W$ of MV-algebras with square roots. In addition, each representable symmetric pseudo MV-algebra with square root is contained in $\mathcal V$.
\end{prop}

\begin{proof}
According to Proposition \ref{3.2.0}(8), we know that $\mathcal W\s \mathcal V$.

First, assume that $M$ is a linearly ordered pseudo MV-algebra with a square root $r$. If $r(0)=0$, then $M$ is a Boolean algebra and $r$ is an identity map, which means inequality (\ref{Eq8.5}) holds. Otherwise, by \cite[Thms 5.1, 5.6]{DvZa3}, $M$ is two-divisible and symmetric with a square root $r(x)=(x+u)/2$ for all $x\in M$.  We have
\begin{eqnarray*}
r(x)\odot r(y)&=&((x+u)/2-u+(y+u)/2)\vee 0=(x+y)/2,\\
r(x\odot y)&=&(((x-u+y)\vee 0)+u)/2=((x+y)\vee u)/2=(x+y)/2\vee u/2, \mbox{ since $M$ is a chain}\\
&=& (r(x)\odot r(y))\vee r(0).
\end{eqnarray*}
Therefore, $M\in\mathcal V$ implies $\mathcal W$ is a proper subvariety of $\mathcal V$.

Let $M$ be a symmetric representable pseudo MV-algebra with a square root $r$ and $X$ be the set of all
normal prime ideals of $M$. Consider the subdirect embedding $f:M\to\prod_{P\in X} M/P$ defined by $f(x)=(x/P)_{P\in X}$.
By \cite[Prop 3.9]{DvZa3}, the onto homomorphism $\pi_P\circ f:M\to M/P$ induces a square root $r_P:M/P\to M/P$ defined by
$r_P(x/P)=r(x)/P$ for all $x\in M$. By the first part, $r_P(x/P\odot y/P)=(r_P(x/P)\odot r_P(y/P))\vee r_P(0/P)$.
It follows that
\begin{eqnarray*}
f\big((r(x)\odot r(y))\vee r(0)\big)&=& \big(f(r(x))\odot f(r(y))\big)\vee f(r(0))=\big((\frac{r(x)}{P})_{P\in X}\odot (\frac{r(y)}{P})_{P\in X}\big)\vee (\frac{r(0)}{P})_{P\in X}\\
&=& \big((r_P(\frac{x}{P}))_{P\in X}\odot (r_P(\frac{y}{P}))_{P\in X}\big)\vee (r_P(\frac{0}{P}))_{P\in X}\\
&=& \big((r_P(x/P)\odot r_P(y/P))\vee r_P(0/P) \big)_{P\in X}=\big(r_P(x/P\odot y/P)\big)_{P\in X}\\
&=& \big(r(x\odot y)/P\big)_{P\in X}=f(r(x\odot y)).
\end{eqnarray*}
So, $(r(x)\odot r(y))\vee r(0)=r(x\odot y)$, which entails that $M\in \mathcal V$.
According to Proposition \ref{3.2.0}(8), we know that $\mathcal W\s \mathcal V$.
\end{proof}

We present an example of a linearly ordered pseudo MV-algebra $M\in \mathcal V \setminus \mathcal W$.

\begin{exm}\label{8.5.1}
Let $\mathbb Q$ be the set of all rational numbers and $G=\mathbb Q\lex \mathbb Q\lex \mathbb Q\lex \mathbb Q$.
Consider the following binary operation on $G$:
\begin{eqnarray}
\label{8.5e1} (a,b,c,d)+(x,y,z,w)=(a+x,b+y,c+z,d+w+bz),\quad \forall (a,b,c,d),(x,y,z,w)\in G.
\end{eqnarray}
Similarly to \cite[Page 138, E41]{Anderson}, we can show that $(G;+,(0,0,0,0))$ is a linearly ordered group, where $-(a,b,c,d)=(-a,-b,-c,-d+bc)$.
Clearly, $G$ is not Abelian.
The element $u=(1,0,0,0)$ is a strong unit of $G$: Indeed, if $(a,b,c,d)\in G^+$, then for each integer $n>1+\max\{|a|,|b|,|c|,|d|\}$ we have $nu=(n,0,0,0)>(a,b,c,d)$. In addition, $G$ is two-divisible:
For each $(a,b,c,d)$ consider the element $(a/2,b/2,c/2,(4d-bc)/8)\in G$. Then
$(a/2,b/2,c/2,(4d-bc)/8)+(a/2,b/2,c/2,(4d-bc)/8)=(a,b,c,(4d-bc)/4+bc/4)=(a,b,c,d)$.
On the other hand, $u/2=(1/2,0,0,0)$ and
$u/2+(a,b,c,d)=(1/2,0,0,0)+(a,b,c,d)=(1/2+a,b,c,d)=(a,b,c,d)+(1/2,0,0,0)=(a,b,c,d)+u/2$ which means $u/2\in \C(G)$ (consequently, $u\in \C(G)$).
By \cite[Exm 3.7]{DvZa3}, $M=\Gamma(G,u)$ is a pseudo MV-algebra with a square root $r$ defined by $r(x)=(x+u)/2$ for all $x\in M$.
By Proposition \ref{8.5}, $M\in \mathcal V$, but $M\notin\mathcal W$.
\end{exm}

\begin{lem}\label{8.6.1}
Let $r$ be a square root on a pseudo MV-algebra $(M;\oplus,^-,^\sim,0,1)$ and $S\s M$. The following properties hold:
\begin{itemize}[nolistsep]
\item[{\rm (i)}] If $\bigwedge S$ exists, then $\bigwedge r(S)$ exists and is equal to $r(\bigwedge S)$.
\item[{\rm (ii)}] If $\bigvee S$ exists, then $\bigvee r(S)$ exists and is equal to $r(\bigvee S)$.
\end{itemize}
\end{lem}

\begin{proof}
(i) By Proposition \ref{3.2.0}(2), for each $s\in S$, we have $r(\bigwedge S)\leq r(s)$. Let $x\in M$ be a lower bound for
$r(S)$. Then $x\leq r(s)$ for all $s\in S$ and so $x\odot x\leq r(s)\odot r(s)=s$. It follows that $x\odot x\leq \bigwedge S$, which implies that
$x\leq r(\bigwedge S)$ (by (Sq2)). Hence, $r(\bigwedge S)$ is the greatest lower bound of $r(S)$, that is, $\bigwedge r(S)=r(\bigwedge S)$.

(ii) Clearly, $r(\bigvee S)$ is an upper bound for the set $r(S)$. Let $b\in M$ be such that $r(s)\leq b$ for all $s\in S$.
Then $b\in [r(0),1]$. According to Proposition \ref{3.2.0}(11), we know that $r(M)=[r(0),1]$, so there exists $a\in M$ such that
$r(a)=b$ which implies that $r(s)\leq r(a)$ for all $s\in S$.
From Proposition \ref{3.2.0}(7), we get that $r(a\wedge x)=r(a)\wedge r(x)=r(x)$, whence $a\wedge x=x$ and so $s\leq a$ for all $s\in S$
(since $r$ is a one-to-one ordered-preserving map). Hence, $\bigvee S\leq a$ and $r(\bigvee S)\leq r(a)=b$.
Therefore, $r(\bigvee S)$ is the least upper bound of the set $r(S)$.
\end{proof}

Recall that, for each Boolean element $a$ of a pseudo MV-algebra $(M;\oplus,^-,^\sim,0,1)$, the set $[a,1]$ is closed under $\oplus$ and $([a,1];\oplus_a,^{-a},^{\sim a},a,1)$ is a pseudo MV-algebra, where $x^{-a}=x^-\vee a$, $x^{\sim a}=x^\sim\vee a$, and $x\oplus_a y=x\oplus y$, $x,y\in [a,1]$. In addition,
for all $x,y\in [a,1]$, $a=a\odot a\leq x\odot y$, so $x\odot_a y:=(x\odot y)\vee a=x\odot y$.

\begin{prop}\label{8.6.2}
Let $(M;\oplus,^-,^\sim,0,1)$ be a pseudo MV-algebra with a square root $r$ such that the element $a=\bigvee_{n\in \mathbb N}r^n(0)$ exists.
Then the pseudo MV-algebra $([a,1];\oplus_a,^{-a},^{\sim a},a,1)$ is a Boolean algebra.
\end{prop}

\begin{proof}
Set $a=\bigvee_{n\in \mathbb N}r^n(0)$. Since $r^{n}(0)\leq r^{n+1}(0)$ for all $n\in\mathbb N$, by Lemma \ref{8.6.1}, we have
$r(a)=r(\bigvee_{n\in \mathbb N}r^n(0))=\bigvee_{n\in\mathbb N}r^{n+1}(0)=\bigvee_{n=2}^{\infty}r^{n}(0)=\bigvee_{n\in \mathbb N}r^n(0)=a$.
Thus $a=r(a)\odot r(a)=a\odot a$, that is $a\in\B(M)$, consequently by the remark mentioned  just before this proposition, $([a,1];\oplus_a,^{-a},^{\sim a},a,1)$ is a pseudo MV-algebra.
Clearly, $r([a,1])\s [a,1]$, so by the note just before this proposition, $r$ is a square root on the pseudo MV-algebra $[a,1]$.
Now, $r(a)=a$, and \cite[Thm 3.8]{DvZa3} imply that $r=\id_{[a,1]}$, so that $([a,1];\oplus_a,^{-a},^{\sim a},a,1)$ is a Boolean algebra.
\end{proof}

From Proposition \ref{8.6.2}, we get that if $(M;\oplus,^-,^\sim,0,1)$ is a $\sigma$-complete pseudo MV-algebra with a square root $r$, then $[\bigvee_{n\in \mathbb N}r^n(0),1]\s \B(M)$.

For example, let $M_1=M_2=M_3=[0,1]$ be the MV-algebra of the unit real interval and $M_2=M_4=\{0,1\}$.
Consider the MV-algebra $M=\prod_{i=1}^5M_i$.
Define $r_i:M_i\to M_i$ by $r_i(x)=(x+1)/2$, $i=1,2,3$ and $r_2=r_4=\id_{\{0,1\}}$.
The mapping $r=(r_1,r_2,r_3,r_4,r_5)$ is a square root on $M$. We have $a:=\bigvee_{n\in\mathbb N}r^n(0)=(1,0,1,0,1)$
and in the MV-algebra $M$, we have $[a,1]=\{a,(1,1,1,0,1),(1,0,1,1,1),(1,1,1,1,1)\}$, which is a Boolean algebra.

Moreover, in a pseudo MV-algebra $M$ with a square root $r$, we have $\bigvee_{n\in\mathbb N}r^n(0)=1$ \iff
$|U_r|=1$ where $U_r:=\{x\in M\mid r^n(0)\leq x,~\forall n\in\mathbb N\}$. We note if $I$ is a normal ideal of $M$ with a square root $r$, then $r/I:M/I \to M/I$, defined by $r/I(x/I)=r(x)/I$, $x/I\in M/I$, is a square root on $M/I$, \cite[Cor 3.10]{DvZa3}.

In \cite[Prop 5.5(vi)]{DvZa3}, we proved that if $r$ is a square root on a pseudo MV-algebra $(M;\oplus,^-,^\sim,0,1)$, then $r(0)\oplus x=x\oplus r(0)$ for all $x\in M$.
We show that $r(0)\odot x=x\odot r(0)$ for all $x\in M$.

\begin{prop}\label{ns1}
Let $(M;\oplus,^-,^\sim,0,1)$ be a pseudo MV-algebra with a square root $r$. Then $r(0)\odot x=x\odot r(0)$ for all $x\in M$.
\end{prop}

\begin{proof}
If $M$ is a Boolean algebra, then the proof is evident.
Choose $x\in M$.

(i) If $M$ is strict, then by \cite[Prop 3.5(vi)]{DvZa3},
$x\odot r(0)=(r(0)^-\oplus x^-)^\sim=(r(0)\oplus x^-)^\sim=(x^-\oplus r(0))^\sim=(x^-\oplus r(0)^-)^\sim=r(0)\odot x$.

(ii) If $M$ is neither strict nor a Boolean algebra, then $v=r(0)^-\odot r(0)^-\neq 0,1$. According to \cite[Thm 4.3]{DvZa3},
$M\cong [0,v]\times [0,v^-]$, where $[0,v]$ is a Boolean algebra and $[0,v^-]$ is a strict pseudo MV-algebra.
Let $r_1:[0,v]\to [0,v]$ and $r_2:[0,v^-]\to [0,v^-]$ be square roots, where $r_1$ is the identity map on $[0,v]$ and $r_2$ is strict.
Define $s:M\to M$ by $s(x)=r_1(x\wedge v)\vee r_2(x\wedge v^-)$.
Then $s(x)\odot s(x)=(r_1(x\wedge v)\vee r_2(x\wedge v^-))\odot (r_1(x\wedge v)\vee r_2(x\wedge v^-))=
(x\wedge v)\vee \big(r_1(x\wedge v)\odot r_2(x\wedge v^-)\big)\vee (x\wedge v^-)=(x\wedge v)\vee 0\vee (x\wedge v^-)=x$.

In addition, if $y\in M$ such that
$y\odot y\leq x$, then
$((y\wedge v)^2\vee (y\wedge v^-)^2)=((y\wedge v)\vee (y\wedge v^-))\odot ((y\wedge v)\vee (y\wedge v^-))\leq (x\wedge v)\vee (x\vee v^-)$, it follows that
$(y\wedge v)^2\leq (x\wedge v)$ and $(y\wedge v^-)^2\leq (x\wedge v^-)$, since $v\wedge v^-=0$.
So, $y\wedge v\leq r_1(x\wedge v)$ and $z\wedge v^-\leq r_2(x\wedge v^-)$, consequently, $y=(y\wedge v)\vee (y\wedge v^-)\leq r_1(x\wedge v)\vee r_2(x\wedge v^-)=s(x)$.
Hence, $s$ is a square root on $M$, so $s=r$.
We have $x\odot r(0)=x\odot s(0)=((x\wedge v)\vee (x\wedge v^-))\odot (r_1(0)\vee r_2(0))=(x\wedge v^-)\odot r_2(0)$.
By (i),  $(x\wedge v^-)\odot r_2(0)=r_2(0)\odot(x\wedge v^-)$.
In a similar way, $r_2(0)\odot(x\wedge v^-)=r(0)\odot x$. Therefore,
$x\odot r(0)=(x\wedge v^-)\odot r_2(0)=r_2(0)\odot(x\wedge v^-)=r(0)\odot x$.
\end{proof}

\begin{cor}
Let $x$ be an element of a pseudo MV-algebra $M$ with a square root $r$. The following statements hold:
\begin{itemize}[nolistsep]
\item[{\rm (i)}] $r(x^n)=r(x)^n\vee r(0)$ for all $n\in\mathbb N$.
\item[{\rm (ii)}] $\bigvee_{n\in\mathbb N}r(x^n)$ exists \iff $\bigvee_{n\in\mathbb N}r(x)^n$ exists.
\end{itemize}
\end{cor}

\begin{proof}
(i) For $n=1$, the proof is clear, since $r(0)\leq r(x)$. For $n=2$, the statement was proved in \cite[Prop 3.3(10)]{DvZa3}.
Let for $2\leq n$ we have $r(x^n)=r(x)^n\vee r(0)$.
By \cite[Prop 3.3(10)]{DvZa3}, $r(x^{n+1})\leq (r(x)\odot r(x^{n}))\vee r(0)$.
Then,
\begin{align*}
(r(x)\odot r(x^{n}))\odot (r(x)\odot r(x^{n}))&=r(x)\odot \big((r(x)^n\vee r(0))\odot r(x)\big)\odot r(x^n)\\
&= r(x)\odot \big((r(x)^n\odot r(x))\vee (r(0)\odot r(x))\big)\odot r(x^n)\\
&= r(x)\odot \big((r(x)\odot r(x)^n)\vee (r(x)\odot r(0))\big)\odot r(x^n)\mbox{, by Proposition \ref{ns1}}\\
&= r(x)\odot r(x)\odot (r(x)^n\vee r(0))\odot r(x^n)\\
&= x\odot r(x^n)\odot r(x^n)=x^{n+1},
\end{align*}
so, by (Sq2), $r(x)\odot r(x^{n})\leq r(x^{n+1})$. Clearly, $r(0)\leq r(x^{n+1})$, whence $(r(x)\odot r(x^{n}))\vee r(0)\leq r(x^{n+1})$.
Therefore,
\begin{eqnarray*}
r(x^{n+1})&=& (r(x)\odot r(x^{n}))\vee r(0)=\Big(r(x)\odot \big( r(x)^n\vee r(0) \big) \Big)\vee r(0)\\
&=& r(x)^{n+1}\vee (r(x)\odot r(0)) \vee r(0)\\&=& r(x)^{n+1}\vee r(0).
\end{eqnarray*}
Therefore, $r(0)\vee r(x)^{n}=r(x^{n})$ for all $n\in\mathbb N$

(ii) It follows from part (i).
\end{proof}

\begin{thm}\label{semisimple}
Let $M=\Gamma(G,u)$ be a semisimple MV-algebra with a square root $r$.
Then $r$ is strict if and only if $\bigvee_{n\in\mathbb N}r^n(0)=u$.
\end{thm}

\begin{proof}
The proof is clear for $|M|=1$, so we assume that $M$ is non-degenerate. Suppose that $r$ is a strict square root on $M$ and $Y=\text{MaxI}(M)$ is the set of all maximal ideals of $M$.
Then $r(0)=u-r(0)$, and so $r(0)=u/2$.
Since $M$ is semisimple, $\bigcap Y=\{0\}$. Consider the natural embedding $f:M\to \prod_{I\in Y}M/I$.
For each $I\in Y$, the MV-algebra $M/I$ is isomorphic to a subalgebra of the unit real interval $[0,1]$. In addition, $|M/I|>2$ and $r_I:M/I\to M/I$ defined by $r_I(x/I)=r(x)/I$ is a strict square root on the linearly ordered MV-algebra $M/I$. By \cite[Thm 5.1]{DvZa3},
for any $n\in\mathbb N$, $r_I^n(0/I)=(2^{n}-1)u/2^n$. Since $M/I$ is isomorphic to a subalgebra of $[0,1]$, we have $\bigvee_{n\in\mathbb N}r_I^n(0/I)=\bigvee_{n\in\mathbb N}(2^{n}-1)u/2^n=1$.
We show that $\bigvee_{n\in\mathbb N}r^n(0)=u$.
Let $a\in M$ be an upper bound for the set $\{r^n(0)\mid n\in\mathbb N\}$. We have
\begin{eqnarray*}
(a/I)_{I\in Y}=f(a)&\ge& f(r^n(0))=f((2^n-1)u/2^n)=(((2^n-1)u/I)/2^n)_{I\in Y}.
\end{eqnarray*}
Hence, $a/I=u/I$ for all $I\in Y$ and so $a=u$.

Conversely, let $\bigvee_{n\in\mathbb N}r^n(0)=u$. We claim that $r$ is strict. If $B$ is a Boolean algebra, by \cite[Thm 3.8]{DvZa3}, $r(0)=0$, so that
$u=\bigvee_{n\in\mathbb N}r^n(0)=0$, which is absurd. Otherwise, by \cite[Thm 2.21]{Hol},
$M\cong [0,v]\times [0,v']$, where $v=r(0)'\odot r(0)'$, $[0,v]$ is a Boolean algebra and $[0,v']$ is a strict MV-algebra.
According \cite[Thm 5.3]{DvZa3}, $r(x)=(x\wedge v)\vee ((x\wedge v')+v')/2$ for all $x\in M$ and so
$r(0)\leq v'$. In a similar way, $r^2(0)=r(r(0))=(r(0)\wedge v)\vee (r(0)+v')/2\leq v'$ and $r^{n}(0)\leq v'$ for all $n\in\mathbb N$.
From $\bigvee_{n\in\mathbb N}r^n(0)=u$ we get $u\leq v'$ which means $u=v'$. Therefore, $M$ is strict.
\end{proof}

\begin{prop}\label{H-root0}
Let $r$ be a square root on a representable pseudo MV-algebra $(M;\oplus,^-,^\sim,0,1)$. Then
\begin{itemize}
\item[{\rm (i)}] $I:=\{x\in M\mid x\le r^n(0)^-,~\forall n\in\mathbb N\}$ is an ideal of $M$.
\item[{\rm (ii)}] If $M$ is linearly ordered, then $I$ is a normal and maximal ideal.
\item[{\rm (iii)}] If $I$ is a normal ideal of $M$, then $|U_{r/I}|=1$.
\end{itemize}
\end{prop}

\begin{proof}
(i) Since $r(0)\leq r^2(0)\leq\cdots\leq r^n(0)\leq \cdots$, we have
$r(0)^-\geq r^2(0)^-\geq\cdots\geq r^n(0)^-\geq \cdots$.
Clearly, $x\leq y\in I$ implies that $x\in I$.
Let $x,y\in I$. Then $x,y\leq r^n(0)^-$ for all $n\in\mathbb N$. Choose $m\in\mathbb N$.
$x\oplus y\leq r^{m+1}(0)^-\oplus r^{m+1}(0)^-=(r^{m+1}(0)\odot r^{m+1}(0))^-=r^{m}(0)^-$. Thus, $x\oplus y\in I$. Therefore,
$I$ is an ideal of $M$.

(ii) If $|M|=2$, then $M=\{0,1\}$, $r=\id_M$, and clearly, $I$ is normal.
Let $3\leq |M|$ and $M=\Gamma(G,u)$, where $(G,u)$ is a unital $\ell$-group.
By \cite[Thm 4.3]{DvZa3}, $r$ is strict. \cite[Thm 5.6]{DvZa3} implies that
$u/2$ exists, $u/2,u\in \C(G)$, and $M$ is symmetric. Moreover, $r(x)=(x+u)/2$ for each $x \in M$ (by \cite[Thm 4.3]{DvZa3}).
Since for each $x\in M$, $r(x^-)^-=u-(u+(u-x))/2=x/2$, then $r(0)^-=u/2$,
$r^2(0)^-=r(r(0)^-)^-=u/4,\ldots, r^{n+1}(0)^-=r((u/2^n)^-)^-=u/2^{n+1}$.
Set $H:=I\cup -I$. Then $H$ is an $\ell$-ideal (a normal convex $\ell$-subgroup) of $G$.

(1)  If $0\geq x\geq y\in -I$, then by (i),
$0\leq -x\leq -y\in I$ and so $x\in -I$. Now, since $G$ is linearly ordered, by (i) $H=I\cup -I$ is convex.

(2) Let $x,y\in I$. We can assume $x\leq y$. Then $x+y\leq y+y$.
If $u\leq y+y$, then $r(0)^- =u/2\leq y$, which contradicts with $y\in I$.
Thus $y+y<u$ and so $y+y=y\oplus y\in I\s H$. If $x,y\in -I$, then
$-x,-y\in I$, so $-(y+x)=-x+-y\in I$ and consequently $y+x\in -I\s H$.

(3) Let $x\in I$ and $y\in -I$. If $0\leq x+y$, then $0\leq x+y\leq x$ and (1) imply that $x+y\in I\s H$. Otherwise,
$y+x\leq 0$, then similarly by (1), $y\leq y+x\leq 0$ implies that $y+x\in -I\s H$.

(4) Let $g\in G$ and $x\in I$. We claim that $g+x-g\in I$.
If $g+x-g>u$, then $x\geq -g+u+g=u-g+g=u$ implies that $x\notin G^+\setminus I$.
Clearly, $0\leq g+x-g$ (since $g+x-g<0$ implies that $x\leq 0$ and so $x=0$). Consequently, $g+x-g=0$.
Thus $0\leq g+x-g\leq u$. If $g+x-g\notin H$, there exists $n\in\mathbb N$ such that $r^{n}(0)^-\leq g+x-g$.
It follows that $g+2^nx-g=2^n(g+x-g)\geq 2^nr^n(0)^-=2^n(u/2^n)=u$ and so
$2^nx\geq -g+u+g=u$, which contradicts with $x\in I$. Similarly, we can show that for each $g\in G$ and $x\in -I$,
$g+x-g\in -I$. Hence, $H$ is a normal ideal (convex $\ell$-subgroup) of $G$, and so is its corresponding ideal
in the pseudo MV-algebra $M$, that is, $I=H\cap M$ is a normal ideal of $M$.

On the other hand, if $J$ is an ideal of $M$ properly containing $I$, then for each $x\in J\setminus I$, there exists $n\in\mathbb N$ such that
$r^n(0)^-< x$. It follows from the first step of this part that $u/2^n\leq x$ and so $u\in J$, which means $J=M$. Therefore,
$I$ is a maximal ideal of $M$.

(iii) Let $I$ be normal. Consider the square root $r/I:M/I\to M/I$ defined by $r/I(x/I)=r(x)/I$. Choose $x/I\in U_{r/I}$.
Then $r_{r/I}^n(0/I)\leq x/I$ for all $n\in\mathbb N$ and so $r^n(0)/I\leq x/I$ for all $n\in\mathbb N$. It follows that
$r^n(0)\odot x^-\in I$. For each $n\in \mathbb N$ we have
$r^n(0)^-\vee x^-=r^n(0)^-\oplus(r^n(0)\odot x^-)\in I$, which means $x^-\in I$. Hence $x/I=1/I$, that is $|U_{r/I}|=1$.
\end{proof}

From the note right before Proposition \ref{H-root0}, it follows that if $M$ is a linearly pseudo MV-algebra with a square root $r$, then $\bigvee\{r^n(0)/I\mid n\in\mathbb N\}=1/I$. In addition, Proposition \ref{8.6.2} implies that if $\bigvee\{r^n(0)\mid n\in \mathbb N\}$ exists, then it is equal to $1$.

In the sequel, we characterize strongly $(H,1)$-perfect MV-algebras with square roots.

\begin{prop}\label{H-perfect1}
Let $M=\Gamma(H\lex G,(1,0))$ be a strongly $(H, 1)$-perfect pseudo MV-algebra with a square root $r$.
Then $M$ satisfies precisely one of the following statements:
\begin{itemize}[nolistsep]
\item[{\rm (i)}]  If $H\cong \mathbb Z$, then $|G|=1$ and $r(0)=0$. The converse holds, too.
\item[{\rm (ii)}] If $r$ is strict, then $H$ is two-divisible. The converse holds, too.
\end{itemize}
\end{prop}

\begin{proof}
Suppose that $M$ has a square root $r$.
By \cite[Thm 3.4]{Dvu4}, $I:=\{0\}\times G^+$ is a normal and maximal ideal of $M$, and $M/I\cong \Gamma(H,1)$. Due to \cite[Prop 3.9]{DvZa3}, there is a square root $t:M/I\to M/I$ defined by $t(x/I)=r(x)/I$ for all $x\in M$.
Consequently, $\Gamma(H,1)$ has a square root $s$.
Since $\Gamma(H,1)$  is linearly ordered and symmetric, \cite[Thm 4.5]{DvZa3} implies $|\Gamma(H,1)|=2$ or
$t$ is strict.

(i) If $|\Gamma(H,1)|=2$ or equivalently $H\cong \mathbb Z$, then $M=(\{0\}\times G^+)\cup (\{1\}\times G^-)$.
Let $r((0,0))=(a,b)$, then $(0,0)=(a,b)\odot (a,b)=((2a,2b)-(1,0))\vee (0,0)=(2a-1,2b)\vee (0,0)$ and so $a=0$.
Clearly, for each $b\leq c\in G^+$ we have $(0,c)\odot (0,c)=(0,0)$, whence by (Sq2), $(0,c)\leq r((0,0))=(0,b)$.
Thus, $b$ is the top element of $G$.
In a similar way, we can show that $r((0,g))=(0,b)$, where $g\in G^+$. Injectivity of $r$ implies that $|G|=1$.
The proof of the converse is straightforward because $|G|=1$ implies that $|M|=2$. Hence, $M$ is a Boolean algebra, and so
$\id_M$ is the only square root on $M$.

(ii) Let $r$ be strict. By (i), $H$ is not isomorphic to $\mathbb Z$, so $2< |\Gamma(H,1)|$, and $s$ is strict.
Due to Theorem \ref{8.1}, $H$ is two-divisible.
Conversely, let $H$ be two-divisible. Then $1/2\in H$.
Thus $(1/2,0)\in M$. Clearly, $(1/2,0)\odot (1/2,0)=(0,0)$. For each $(a,b)\in M$, $(a,b)\odot (a,b)\leq (0,0)$ implies that
$(a,b)-(1,0)+(a,b)\leq (0,0)$ consequently, $(2a,2b)=2(a,b)\leq (1,0)=2(1/2,0)$.
If $2a< 1$, then clearly $(a,b)\leq (1/2,0)$. If $2a=1$, then $a=1/2$ and $2b\leq 0$. From \cite[Prop 3.6]{Dar}, it follows that $b\leq 0$ and so $(a,b)\leq (1/2,0)$. That is, $r((0,0))=(1/2,0)$.
Therefore, $r$ is strict.
\end{proof}

\begin{thm}\label{H-perfect2}
Let $M=\Gamma(H\lex G,(1,0))$ be a strongly $(H, 1)$-perfect pseudo MV-algebra such that $G$ enjoys unique extraction of roots and  $2<|M|$.
If $M$ has a square root, then $H$ and $G$ are two-divisible. In general, if $G$ does not enjoy unique extraction of roots, then for each
$g\in G$, the set $\{x\in G\mid 2x=g\}$ has a top element in $G$.
\end{thm}

\begin{proof}
Suppose that $r$ is a square root on $M$. Since $2<|M|$, by Proposition \ref{H-perfect1}, $r$ is strict and $H$ is two-divisible, so $1/2\in H$. Choose $g\in G$.
Then $(1/2,g)\in M$. Let $r((1/2,g))=(a,b)\in M$. Then
\begin{eqnarray*}
(1/2,g)=(a,b)\odot (a,b)=\big((a,b)-(1,0)+(a,b)\big)\vee (0,0)= (2a-1,2b)\vee (0,0).
\end{eqnarray*}
If $2a-1<0$, then $(2a-1,2b)\vee (0,0)=(0,0)\neq (1/2,g)$ and if $2a-1=0$, then $(2a-1,2b)\vee (0,0)=(0,2b)\neq (1/2,g)$.
It follows that $0<2a-1$ and so $(2a-1,2b)\vee (0,0)=(2a-1,2b)$ that implies $2a-1=1/2$ and $2b=g$. That is $a=3/4$, $2b=g$, and
$G$ is two-divisible. Since $G$ enjoys unique extraction of roots, we have $b=g/2$ and so $r((1/2,g))=(3/4,g/2)$.

Now, let $G$ not enjoy the unique extraction of roots. Choose $g\in G$. By the first step of the proof, we know that
$\{x\in G\mid 2x=g\}\neq\emptyset$ and $r(1/2,g)=(3/4,d)$ for some $d\in G$ with $2d=g$.
If $x\in G$ such that $2x=g$, then $(3/4,x)\odot (3/4,x)=(1/2,2x)=(1/2,g)$, so $(3/4,x)\leq (3/4,d)$ which implies that $x\leq d$.
Therefore, $d=\max\{x\in G\mid 2x=g\}$.
\end{proof}

\begin{cor}\label{H-perfect3}
Let $M=\Gamma(H\lex G,(1,0))$ be a strongly $(H,1)$-perfect pseudo MV-algebra with a square root $r$, and let $G$ enjoy unique extraction of roots.
Then
\begin{eqnarray}
r((x,y))=\big(\frac{x+1}{2},\frac{y}{2}\big), \quad\forall (x,y)\in M\setminus\{(a,b)\in M\mid a\neq 0\}.
\end{eqnarray}
\end{cor}

\begin{proof}
Choose $(x,y)\in M$. Let $r((x,y))=(a,b)$. Then $(2a-1,2b)\vee (0,0)=(x,y)$.

(i) If $x=0=y$ then there are two cases, $2a-1<0$ or $2a-1=0=2b$.
By Proposition \ref{H-perfect1},
$r((x,y))=(1/2,0)=((x+1)/2,y/2)$.

(ii) If $x=0$ and $0<y$. Then $(x,y)=(2a-1,2b)\vee (0,0)=(0,2b\vee 0)$. So, $2a-1=x=0$ and $2b\vee 0=y$, which means
 $r((x,y))=((x+1)/2,y/2)=(1/2,y/2)$.

$2a-1=0$, $2b\geq 0$ and $(2a-1,2b)\vee (0,0)=(2a-1,2b)$. It follows that $a=(x+1)/2$ and $b=y/2$. That is, $r(x,y)=((x+1)/2,y/2)$.

(iii) If $x>0$, then  $(2a-1,2b)=(2a-1,2b)\vee (0,0)=(x,y)$ and so $a=(x+1)/2$ and $b=y/2$.  That is, $r(x,y)=((x+1)/2,y/2)$.
\end{proof}

\end{document}